\documentclass[10pt,a4paper,reqno]{article}
\usepackage{graphicx} 
\usepackage[english]{babel} 
\usepackage{csquotes} 
\usepackage{filecontents} 
\usepackage[dvipsnames]{xcolor}
\usepackage{color}
\usepackage[colorinlistoftodos]{todonotes}
\usepackage{tensor}
\usepackage{thmtools}
\usepackage{microtype}

\usepackage[
backend=biber,
hyperref=true,
backref=true,
isbn=false,
doi=true,
url=false,
natbib=true,
eprint=true,
useprefix=true,
maxcitenames=99,
maxbibnames=99,  
maxalphanames=99, 
minalphanames=99,
safeinputenc,
style=alphabetic,
citestyle=alphabetic,
block=space,
datamodel=ext-eprint,
]{biblatex}
\usepackage[
hypertexnames = false,
colorlinks    = true,
citecolor     = teal,
linkcolor     = blue,
urlcolor      = blue,
linktocpage = true,
breaklinks
]{hyperref}
\renewbibmacro{in:}{} 

\DeclareSourcemap{
	\maps[datatype=bibtex]{
		\map{
			\step[fieldsource=pmid, fieldtarget=pubmed]
		}
	}
}

\makeatletter
\DeclareFieldFormat{arxiv}{%
	arXiv\addcolon\space
	\ifhyperref
	{\href{http://arxiv.org/\abx@arxivpath/#1}{%
			\nolinkurl{#1}%
			\iffieldundef{arxivclass}
			{}
			{\addspace\texttt{\mkbibbrackets{\thefield{arxivclass}}}}}}
	{\nolinkurl{#1}
		\iffieldundef{arxivclass}
		{}
		{\addspace\texttt{\mkbibbrackets{\thefield{arxivclass}}}}}}
\makeatother
\DeclareFieldFormat{pmcid}{%
	PMCID\addcolon\space
	\ifhyperref
	{\href{http://www.ncbi.nlm.nih.gov/pmc/articles/#1}{\nolinkurl{#1}}}
	{\nolinkurl{#1}}}
\DeclareFieldFormat{mrnumber}{%
	MR\addcolon\space
	\ifhyperref
	{\href{http://www.ams.org/mathscinet-getitem?mr=MR#1}{\nolinkurl{#1}}}
	{\nolinkurl{#1}}}
\DeclareFieldFormat{zbl}{%
	Zbl\addcolon\space
	\ifhyperref
	{\href{http://zbmath.org/?q=an:#1}{\nolinkurl{#1}}}
	{\nolinkurl{#1}}}
\DeclareFieldAlias{jstor}{eprint:jstor}
\DeclareFieldAlias{hdl}{eprint:hdl}
\DeclareFieldAlias{pubmed}{eprint:pubmed}
\DeclareFieldAlias{googlebooks}{eprint:googlebooks}

\renewbibmacro*{eprint}{%
	\printfield{arxiv}%
	\newunit\newblock
	\printfield{jstor}%
	\newunit\newblock
	\printfield{mrnumber}%
	\newunit\newblock
	\printfield{zbl}%
	\newunit\newblock
	\printfield{hdl}%
	\newunit\newblock
	\printfield{pubmed}%
	\newunit\newblock
	\printfield{pmcid}%
	\newunit\newblock
	\printfield{googlebooks}%
	\newunit\newblock
	\iffieldundef{eprinttype}
	{\printfield{eprint}}
	{\printfield[eprint:\strfield{eprinttype}]{eprint}}}

\usepackage[margin=1.7cm,includehead]{geometry}

\usepackage{amsmath,stmaryrd, upgreek}
\usepackage{amsthm,amssymb}
\usepackage{latexsym}
\usepackage{amscd}
\usepackage{mathrsfs}
\usepackage{url}
\usepackage{mathtools}
\usepackage{doi}

\usepackage[T1]{fontenc}
\usepackage{tikz-cd}
\usepackage{titlesec}

\usepackage[titletoc,toc,title]{appendix}

\makeatletter 
\renewcommand\tableofcontents{%
	\@starttoc{toc}%
	
}
\makeatother

\usepackage{enumerate}
\usepackage{slashed}
\graphicspath{}
\usepackage{fancyhdr} 

\usepackage{changepage}

\newcounter{noteCounter}
\usepackage[nameinlink]{cleveref} 

\newcommand\shorttitle{Solutions and singularities of flows of $\G2$-structures} 
\newcommand\authors{S. Dwivedi \& R. Singhal} 

\newcounter{commentCounter}
\titleformat{\section}    
{\normalfont\large\bfseries\center}{\thesection.}{1em}{}

\makeatletter
\newcommand*{\rom}[1]{\expandafter\@slowromancap\romannumeral #1@}
\makeatother

\newtheorem{theorem}{Theorem}[section]

\newtheorem{lemma}[theorem]{Lemma}
\newtheorem{proposition}[theorem]{Proposition}

\theoremstyle{definition}
\newtheorem{definition}[theorem]{Definition}
\newtheorem{remark}[theorem]{Remark}
\newtheorem*{ack}{Acknowledgments}
\numberwithin{equation}{section}
\def\bR{\mathbb R}

\def\pt{\partial}
\def\del{\nabla}
\def\G2{\mathrm{G}_2}
\def\g2{\varphi}
\def\S7{\mathrm{Spin}(7)}
\def\s7{\Phi}
\def\ddt{\frac{d}{dt}}
\def\ptt{\frac{\partial}{\partial t}}

\def\oep{\overline{\Upsilon}}

\def\cD{\mathcal{D}}

\def\cF{\mathcal{F}}

\def\cL{\mathcal{L}}

\def\cT{\mathcal{T}}

\newcommand{\Vop}{\mathsf V}

\def\Spin7{\mathrm{Spin(7)}}
\def\Ric{\mathrm{Ric}}
\def\Riem{\mathrm{Rm}}
\def\Rm{\mathrm{Rm}}

\def\dots7{\Dot{\Phi}}

\DeclareMathOperator\Div{div}

\DeclareMathOperator\vol{vol}
\DeclareMathOperator\Vol{Vol}
\DeclareMathOperator\tr{tr}

\newcommand\xqed[1]{%
	\leavevmode\unskip\penalty9999 \hbox{}\nobreak\hfill
	\quad\hbox{#1}}
\newcommand\demo{\xqed{$\blacktriangle$}}

\newcommand{\qandq}{\quad\text{and}\quad}

\usepackage{cancel}
\usepackage[all]{xy}
\usepackage{multicol}

\usepackage{charter}

\makeatletter
\def\blfootnote{\xdef\@thefnmark{}\@footnotetext}
\makeatother

\begin{document}

	\title{Solutions and singularities of the Ricci-harmonic flow and Ricci-like flows of $\mathrm{G_2}$-structures}
	\author{Shubham Dwivedi \& Ragini Singhal}
	\date{\today}
	
	\maketitle

	\begin{abstract}
		We find explicit solutions and singularities of the Ricci-harmonic flow of $\G2$-structures, the Ricci-like flows of $\G2$-structures studied in \cite{panos-georgeG2Hilbert} and of the negative gradient flow of an energy functional of $\G2$-structures, on $7$-dimensional contact Calabi-Yau manifolds and the $7$-dimensional Heisenberg group. We prove that the natural co-closed $\G2$-structure on a contact Calabi-Yau manifold as the initial condition leads to an ancient solution of the Ricci-harmonic flow with a finite time Type I singularity, and it gives an immortal solution to the Ricci-like flows with an infinite time singularity which are Type III if the transversal Calabi-Yau distribution is flat, and Type IIb otherwise. The same ansatz give ancient solution to the negative gradient flow of $\G2$-structures. These are the first examples of Type I singularities of the Ricci-harmonic flow and Type IIb and Type III singularities of the Ricci-like flows. We also obtain similar solutions for all the three flows on the $7$-dimensional Heisenberg group. 
	\end{abstract}
	
	\begin{adjustwidth}{0.95cm}{0.95cm}
		\tableofcontents
	\end{adjustwidth}


	\section{Introduction}\label{sec:intro}

	The purpose of this paper is to study explicit solutions and their behaviours for some general flows of $\G2$-structures. Ever since the introduction of the Ricci flow to study deformation of Riemannian metrics by Hamilton \cite{hamilton-3manifolds}, geometric flows have proven to be a powerful analytical tool to study various geometric and topological problems. There have been various proposals of geometric flows in $\G2$-geometry and one of their far-reaching goals is that they would provide a method to search for metrics with holonomy $\G2$ and as a result, Ricci-flat metrics on a spin and oriented $7$-manifold $M$. This corresponds to a varying family $\g2(t)$ of $\G2(t)$-structures, given by non-degenerate $3$-forms $\varphi(t)$ on $M$, with the hope that that it becomes torsion-free in appropriate limits. There have been different such flows in $\G2$-geometry starting with the Laplacian flow of closed $\G2$-structures by Bryant \cite{bryant-remarks}. See \cite{bryant-xu, lotay-wei-gafa, kmt, grigorian-modified, weiss-witt, grigorian-isometric, dgk-isometric, loubeau-saearp, dgk-flows, panos-georgeG2Hilbert, dwivedi-rhf, fino-fowdar} for other flows of $\G2$-structures. 
	
	\medskip
	
	Recall that if $M$ is a $7$-dimensional smooth, oriented and spin manifold then a $\G2$-structure $\varphi$ on $M$ is a $3$-form which is nondegenerate in the sense that it  determines a  metric $g_{\varphi}$ and an orientation $\vol_{\varphi}$. 
	The $\G2$-structure is called \emph{torsion-free} if $\nabla^{g_\varphi}\varphi=0$, where $\nabla^{g_{\varphi}}$ is the Levi-Civita connection of $g_{\varphi}$. Since $g$ is obtained from $\g2$ in a nonlinear way, the equation for $\g2$ being torsion-free is a nonlinear PDE and as a result, finding torsion-free $\G2$-structures with the metrics having full holonomy $\G2$ is a challenging problem. The problem becomes considerably more difficult when $M$ is required to be compact. 
	
	\medskip
	
	The flow of $\G2$-structures we study in this paper is the Ricci-harmonic flow which was introduced by the first author in \cite{dwivedi-rhf}. Let $\g2(t)$ be a family of $\G2$-structures on a compact manifold $M^7$ with some $\G2$-structure $\g2_0$. If $T$ denotes the torsion of the $\G2$-structure $\g2$, $\Ric$ is the Ricci curvature of the underlying metric, $\psi=*_{\g2}\g2$ and $\diamond$ is the operation described in \eqref{eq:diadefn}, then Ricci-harmonic flow is the initial value problem
	\begin{align} 
		\label{eq:rhfeqn} 
		\left\{\begin{array}{rl} 
			& \dfrac{\pt \g2}{\pt t} = \left(-\Ric+3 T^tT -|T|^2g \right) \diamond \g2 + \Div T\lrcorner \psi, \\
			& \g2(0) =\g2_0.
			\tag{RHF}
		\end{array}\right. 
	\end{align}
	Here, $\Div T$ is the divergence of the torsion which is a $1$-form (or a vector field) and is explicitly given in coordinates as $(\Div T)_{i}=\del^{j}T_{ji}$. The flow in \eqref{eq:rhfeqn} is a general flow of $\G2$-structures in the sense that we do not put any conditions on either the initial or evolving $\G2$-structures. This flow was obtained by looking at the Taylor series expansion of a $\G2$-structure $\g2$ in normal coordinates and then computing the "Laplacian of the components of $\g2$" by looking at the second order terms in the expansion. This is a procedure similar to obtaining the Ricci flow of metrics. As such, \eqref{eq:rhfeqn} can be viewed as the heat flow of $\G2$-structures. It is also a natural coupling of the Ricci flow of $\G2$-structures and the \emph{isometric/harmonic} flow of $\G2$-structures studied in \cite{grigorian-isometric, dgk-isometric, loubeau-saearp}. It was shown in \cite[\textsection 3]{dwivedi-rhf} that on compact manifolds, the stationary points of \eqref{eq:rhfeqn} are precisely torsion-free $\G2$-structures. The Ricci-harmonic flow was studied in detail by the first author in \cite{dwivedi-rhf}. It was proved that a unique solution to \eqref{eq:rhfeqn} exists for a short time on a compact manifold starting with any arbitrary $\G2$-structure. Among other results, the following theorem was proved in \cite[Thm. 5.1]{dwivedi-rhf} concerning the maximal existence time of the Ricci-harmonic flow on closed $7$-manifolds.
	
	\begin{theorem}\label{thm:rhflte}
		Let $\g2(t)$ be a solution of the Ricci-harmonic flow on a closed $7$-manifold $M$ on a maximal time interval $[0, \tau)$. We define the following quantity along the flow.
		\begin{align}\label{eq:Lambdadefn}
			\Lambda(t)= \underset{M}{\text{sup}}\  \Lambda(x,t)\ \ \text{where}\ \ \Lambda(x,t)= \left(|\Riem(x,t)|^2+|\del T(x,t)|^2+|T(x,t)|^4\right)^{\frac 12}.
		\end{align}
		Then
		\begin{align}\label{eq:lte1}
			\lim_{t\nearrow \tau} \Lambda(t)=\infty. 
		\end{align}	
		Moreover, the quantity $\Lambda(t)$ blows-up at the following rate,
		\begin{align}\label{eq:Lambdablowuprate}
			\Lambda(t)\geq \dfrac{C}{\tau-t},
		\end{align}
		where $C>0$ is a constant.
	\end{theorem}
	
	Thus, the Ricci-harmonic flow exists on a closed manifold as long as $\Lambda(t)$ remains bounded. Using the blow-up rate on $\Lambda$ in \eqref{eq:Lambdablowuprate}, the type of singularities of the Ricci-harmonic flow can be classified. We recall the following definition from \cite[Defn. 5.3]{dwivedi-rhf}.
	
	\begin{definition}\label{def:sing.types}
		Suppose that $(M^7,\varphi(t))$ is a solution of the Ricci-harmonic flow of $\G2$-structures on a closed manifold on a maximal time interval $[0,\tau)$ and let $\Lambda(t)$ be as in \eqref{eq:Lambdadefn}.
		
		\medskip
		
		\noindent	
		If we have a finite-time singularity, i.e.~$\tau<\infty$, we say that the solution forms  
		\begin{enumerate}
			\item a \emph{Type I singularity} (rapidly forming) if $\sup_{t\in[0,\tau)}(\tau-t)\Lambda(t) <\infty$; and otherwise \label{itemTypeI}
			\item a \emph{Type IIa singularity} (slowly forming) if $\sup_{t\in [0,\tau)}(\tau-t)\Lambda(t)=\infty$.
		\end{enumerate}
		
		\noindent	
		If we have an \emph{infinite-time} singularity, where $\tau=\infty$, then it is 
		\begin{enumerate}
			\item a \emph{Type IIb singularity} (slowly forming) if $\sup_{t\in  [0,\infty)}t\Lambda(t)=\infty$; and otherwise  
			
			\item a \emph{Type III singularity} (rapidly forming) if $\sup_{t\in  [0,\infty)}t\Lambda(t)<\infty$.
		\end{enumerate}
	\end{definition}
	
	As with any geometric flow, its important to have solutions of the flow as well as examples of singularities of any of the type described in \Cref{def:sing.types}. It was shown in \cite[Prop. 3.13]{dwivedi-rhf} that \emph{nearly} $\G2$-\emph{manifolds} are shrinking solutions of the Ricci-harmonic flow. In this paper, we show the existence of solutions of the Ricci-harmonic flow on two classes of manifolds with $\G2$-structures. The first class are $7$-dimensional \emph{contact Calabi-Yau} manifolds and on these, we prove the existence of solutions which are 1) \emph{ancient solutions}, i.e., they exist for all negative times but not \emph{eternal} and, 2) they form Type I singularity described in \cref{itemTypeI}. These are the first examples of such singularities of the Ricci-harmonic flow on closed manifolds and also give new examples of the flow itself. The contact Calabi-Yau manifolds we consider are given by the data $(M^{7},g,\eta, \Upsilon)$, where $(M^7,g)$ is a Sasakian $7$-manifold with Riemannian metric $g$, contact form $\eta$ and transverse K\"ahler form $\omega=d\eta\in\Omega^{1,1}(M)$, and $\Upsilon\in\Omega^{3,0}(M)$ is a  transverse holomorphic volume form; here $(p,q)$ denotes basic bi-degree with respect to the horizontal distribution $\cD=\ker\eta$, see \Cref{subsec:cCYmanifolds}. The second class of manifolds where we prove the existence of explicit solutions to the Ricci-harmonic flow are on the $7$-dimensional Heisenberg group where we show that the solutions are again ancient and have Type I singularities. 
	
	\medskip
	
	The motivation for studying the flows of $\G2$-structures on contact Calabi-Yau manifolds and the $7$-dimensional Heisenberg group stems from the works \cite{lotay-saearp-saavedra} by Lotay--Sá Earp--Saavedra and \cite{bff-coflow} by Bagaglini--Fernández--Fino, respectively. In \cite{lotay-saearp-saavedra}, an explicit ansatz of natural $\G2$-structures on contact Calabi-Yau $7$-manifolds was provided which gave new solutions and singularities for the Laplacian co-flow of co-closed $\G2$-structures and the Laplacian flow of $\G2$-structures. In \cite{bff-coflow}, a similar analysis was done for the Laplacian co-flow and the modified Laplacian co-flow of co-closed $\G2$-structures on the $7$-dimensional Heisenberg group. We use the same ansatz as in \cite{lotay-saearp-saavedra} and \cite{bff-coflow} for our analysis of the Ricci-harmonic flow and Ricci-like flows of $\G2$-structures considered in \cite{panos-georgeG2Hilbert} (see \eqref{eq:pgflw1} for more details). We mention that $\G2$-structures on contact Calabi-Yau $7$-manifolds were also analyzed to study solutions of Grigorian's modified Laplacian co-flow of co-closed $\G2$-structures as well as the dimensional reductions of the Laplacian flow and co-flow considered in \cite{picard-suan} by Sá Earp--Saavedra--Suan in \cite{saearp-saavedra-suan} and solutions to the heterotic $\G2$-system by Lotay--Sá Earp \cite{lotay-saearp}. 
	
	\medskip

	Other geometric flow of $\G2$-structures which we consider in this paper are the flows considered by Gianniotis--Zacharopoulos in \cite{panos-georgeG2Hilbert}. The idea in \cite{panos-georgeG2Hilbert} is to consider a modified $\G2$-Einstein-Hilbert functional on the space $\Omega^3_+(M)$ of $\G2$-structures on $M$ and look at its variation. The $\G2$-Einstein-Hilbert functional is given by 
	\begin{align}\label{eq:G2EinsteinHilbert}
		\cF(\g2)= \int_M \left(\frac 16 R - \frac 13|T|^2 -\frac 16 (\tr T)^2 \right)\vol_{\g2},
	\end{align}
	where $R$ is the scalar curvature, $T$ is the torsion $2$-tensor of $\g2$ and $\tr T$ is the metric trace of the torsion. All the norms and trace are take with respect to the metric induced by $\g2$. The flows of $\G2$-structures obtained by the variation of $\cF$ are \cite[eqs. (4.20) and (4.21)]{panos-georgeG2Hilbert}
	\begin{align}
		\ptt \g2 &= \left(-\Ric -\frac 23 (T\circ (\Vop{T}\lrcorner \g2))_{\text{sym}} + \tr TT_{\text{sym}} \right) \diamond \g2 + \left(\Div T +\frac 13\tr T \Vop{T} - \frac 13 T^t(\Vop{T}) \right)\lrcorner \psi, \label{eq:pgflw1} \\
		\ptt \g2 &= \left(-\Ric -\frac 23 (T\circ (\Vop{T}\lrcorner \g2))_{\text{sym}} + \tr TT_{\text{sym}} + \frac 13 \left(|T|^2-\frac 13 |\Vop{T}|^2 \right)g \right) \diamond \g2 + \left(\Div T +\frac 13\tr T \Vop{T} - \frac 13 T^t(\Vop{T}) \right)\lrcorner \psi. \label{eq:pgflw2}
	\end{align}
	Here $\Vop{T}$ is the $7$-dimensional part of the torsion which can be identified with a vector field on $M$, $T_{\text{sym}}$ denotes the symmetric part of the torsion and $T^t$ is the transpose of the $2$-tensor $T$. The flow in \eqref{eq:pgflw1} is obtained by lookng at the variation of the $\G2$-Einstein-Hilbert functional and the flow in \eqref{eq:pgflw2} is obtained by looking at a perturbed version of the integrand in the variation (see \cite{panos-georgeG2Hilbert} for more details).

	We find explicit solutions for the flows \eqref{eq:pgflw1}-\eqref{eq:pgflw2} on contact Calabi-Yau $7$-manifolds which are \emph{immortal} solutions, i.e., they exist for all positive times. These are first such examples of this kind for the flow. At the moment, we are lacking a result like \Cref{thm:rhflte} for the flow \eqref{eq:pgflw1}, but given their similarity with the Ricci-harmonic flow \eqref{eq:rhfeqn}, it is expected that a quantity similar to $\Lambda(t)$ will govern the maximal existence time of the flows of variation of the $\G2$-Einstein-Hilbert functional and so our solutions might provide the first examples of Type IIb and Type III infinite time singularities of the flow. We also obtain solutions to these flows on the $7$-dimensional Heisenberg group. 
	
	\medskip
	
	The final geometric flow of $\G2$-structures which we consider in the paper is the negative gradient flow of the natural energy functional $\g2\mapsto \frac 12\int |T|^2\vol$ on the space of all $\G2$-structures\footnote{We thank Panagiotis Gianniotis for suggesting to analyze this flow as well.}. This is essentially the heat flow of $\G2$-structures studied by Weiss--Witt in \cite{weiss-witt} and was also studied as a special case in \cite{dgk-flows}. The flow is explicitly given by (see \Cref{subsec:ngf} for more details)
	\begin{align*}
\ptt \g2 = \left(-\Ric -\frac 12\cL_{\Vop{T}}g - \frac 12|T|^2g+(\tr T)T_{\text{sym}}-T^2_{\text{sym}}+TT^t-(T(PT))_{\text{sym}}\right) + \Div T\lrcorner \psi.
	\end{align*}
We again find explicit solutions for the flow on contact Calabi-Yau manifolds which are \emph{ancient solutions}. These are first such examples of this kind for the negative gradient flow. We also obtain solutions of these flows on the $7$-dimensional Heisenberg group.	
	
	\medskip
	
	We emphasize that since the Ricci-harmonic flow \eqref{eq:rhfeqn}, the Ricci-like flows in \eqref{eq:pgflw1}-\eqref{eq:pgflw2} and the negative gradient flow of the energy functional are general flows of $\G2$-structures, i.e., no \emph{a priori}, conditions on either the initial or evolving $\G2$-structures are needed for the flows, it does not really matter whether the ansatz considered has special form and this is a fact which is different from the considerations in \cite{lotay-saearp-saavedra}, \cite{bff-coflow} and \cite{saearp-saavedra-suan}. Of course, the ansatz we consider in \eqref{eq:phit}, \eqref{eq:psit} and \eqref{eqn:phit_h7} are always co-closed $\G2$-structures so the analysis of the flows become easier. However, we expect that more complicated ansatz (not necessarily on contact Calabi-Yau manifolds) might also provide non-trivial solutions and singularities for the flows considered in this paper. A particularly interesting problem would be to understand the dimensional reductions of the flows considered in this paper to $4$ and $6$-dimensional manifolds with $\mathrm{SU}(2)$ and $\mathrm{SU}(3)$-structures, respectively. 
	
	\medskip
	
	The outline of the paper is as follows. We discuss some preliminaries on $\G2$-structures, in general, in \Cref{subsec:G2geometry} and on contact Calabi-Yau manifolds in \Cref{subsec:cCYmanifolds}. We describe our families of $\G2$-structures which will be the solutions to \eqref{eq:rhfeqn}, \eqref{eq:pgflw1} and \eqref{eq:neggradflow} in \Cref{sec:ansatz} where we also compute the Ricci tensor, torsion tensor and other quantities required to analyze the flows. The main results of the paper are in \Cref{sec:mainresults} and \Cref{sec:solonheisenberg}. These, in the contact Calabi-Yau case are, \Cref{thm:RHFcCY} where we find explicit solutions and Type I singularities of \eqref{eq:rhfeqn}, \Cref{thm:pgflowcCY} where we find solutions of \eqref{eq:pgflw1} and \eqref{eq:pgflw2} and \Cref{thm:ngflowcCY} where we do the same for the negative gradient flow. On the Heisenberg group, the general ordinary differential equations for the functions in the ansatz in \eqref{eqn:phit_h7} which provide solutions to the Ricci-harmonic flow are obtained in \Cref{thm:RHFheisenberggeneral} and explicit solutions are described in \Cref{thm:RHFheisenberg}. Similarly, for the Ricci-like flows, the general ODEs are described in equations \eqref{eq:pgflw1heigen}, \eqref{eq:pgflw2heigen} for the Ricci-like flows in \cite{panos-georgeG2Hilbert} and in \eqref{eq:ngflw1heigen} for the negative gradient flow. 
	
	\ack{A part of this work was done during the BIRS-CMI workshop "26w5603: Connections among Spin Geometry, Minimal Surfaces and Relativity" and both the authors would like to thank the BIRS for the opportunity, the Chennai Mathematical Institute for providing stimulating working conditions and the participants of the workshop for excellent scientific discussions. The authors would like to thank Panagiotis Gianniotis, Jason Lotay  and Simon Salamon for useful discussions. The first author acknowledges support by the Deutsche Forschungsgemeinschaft (DFG, German Research Foundation) under Germany’s Excellence Strategy – EXC 2121 "Quantum Universe" – 390833306. The second author is funded by the Deutsche Forschungsgemeinschaft (DFG, German Research Foundation) under Germany’s Excellence Strategy EXC 2044 –390685587, Mathematics Münster: Dynamics–Geometry–Structure.}

	\section{$\G2$-geometry and contact Calabi-Yau manifolds}\label{sec:G2andcCY}
	
	\subsection{$\G2$-geometry}\label{subsec:G2geometry}
	Let $M^7$ be an oriented smooth manifold. A $\G2$-structure on $M$ is the reduction of the structure group of the frame bundle $Fr(M)$ from the Lie group $\mathrm{GL}(7, \bR)$ to the Lie group $\G2$. Equivalently, a $\G2$-structure is given by a $3$-form $\g2\in \Omega^3(M)$ which is nondegenerate, which means that it determines a Riemannian metric $g_\g2$ and an orientation which are given by the formula
	\begin{align}\label{eq:metricfromphi}
		(X\lrcorner \g2)\wedge (Y\lrcorner \g2)\wedge \g2 = 6 g_{\g2}(X,Y)\vol_{\g2},\ \ \ X,\ Y\in \Gamma(TM).
	\end{align}
	The metric and orientation determines the Hodge star operator and we denote by $\psi=*_{\g2}\g2$. The space of differential forms decompose further into irreducible $\G2$-representation. We have the following orthogonal decomposition with respect to $g$,
	\begin{align*}
		\Omega^2=\Omega^2_7\oplus \Omega^2_{14}\ \ \ \ \ \text{and}\ \ \ \ \ \Omega^3=\Omega^3_1\oplus \Omega^3_7\oplus \Omega^3_{27},
	\end{align*}
	where $\Omega^k_l$ has pointwise dimension $l$. Let $\gamma \in \Omega^k$. Given any $2$-tensor $A$ on $M$, we define
	\begin{align}\label{eq:diadefn}
		(A\diamond \gamma)_{i_1i_2\cdots i_k}=A_{i_1}^{\, \, p}\gamma_{pi_2\cdots i_k}+A_{i_2}^{\, \, p}\gamma_{i_1pi_3\cdots i_k}+\cdots + A_{i_k}^{\, \, p}\gamma_{i_1i_2\cdots i_{k-1}p},	
	\end{align}	
	where we are raising the indices using the underlying metric. So, for instance, $g\diamond \gamma = k\gamma$. The operation $\diamond$ is the infinitesimal action of the group $\mathrm{GL}(7, \bR)$ on the space of differential forms. This operation is the same as the map $i_{\g2}$ in \cite{bryant-remarks} (up to a factor of $\frac 12$) and we will use them interchangeably. If $\cT^2$ denote the space of $2$-tensors, then we get a linear map 
	\begin{align*}
		\diamond : \cT^2 \rightarrow \Omega^3.
	\end{align*}
	
	Recall that $\cT^2= C^{\infty}(M)g \oplus S^2_0 \oplus \Omega^2_7 \oplus \Omega^2_{14}$, where $S^2_0$ denote the space of symmetric traceless $2$-tensors. It's easy to prove that $\ker(\diamond) \cong \Omega^2_{14}$ and we have isomorphisms
	\begin{align*}
		C^{\infty}(M)\cong \Omega^3_1,\ \ \ \ \ \Omega^2_7 \cong \Omega^3_7 \cong  \Omega^1,\ \ \ \ \  \ S^2_0\cong \Omega^3_{27}.
	\end{align*}
	As a result, any $3$-form $\sigma\in \Omega^3(M)$ can be described by a pair $(h,X)$ where $h$ is a symmetric $2$-tensor and $X$ a vector field on $M$. We have
	\begin{align*}
		\sigma = h\diamond \g2+ X\lrcorner \psi.
	\end{align*}
	
	
	The torsion of a $\G2$-structure is a $2$-tensor $T$ and is given by
	\begin{align}\label{eq:delphi}
		\del_m\g2_{ijk}= T_{m}^{\ p}\psi_{pijk}.
	\end{align}	
	Since $T$ is a $2$-tensor, we can decompose it further into \emph{intrinsic torsion forms}. These intrinsic torsion forms can be described by looking at the expressions of $d\g2$ and $d\psi$ (see \cite{bryant-remarks} or \cite{flows1} for more details),
	\begin{align}
		\label{eq:dphi_int} d\g2 &= \tau_0\psi+ 3\tau_1\wedge \g2 + *\tau_3, \\
		\label{eq:dpsi_int} d\psi&= 4\tau_1\wedge \psi +  \tau_2\wedge \g2,
	\end{align}
	where $\tau_0\in C^{\infty}(M),\  \tau_1\in \Omega^2_7,\ \tau \in \Omega^2_{14}$ and $\tau_3\in \Omega^3_{27}$. We have an isomorphism between the space $S^2$ of traceless symmetric $2$-tensors and $\Omega^3_{1\oplus 27}$ which is given by a $\G2$-specific map $j_{\g2}:\Omega^3\rightarrow S^2$ which was first described by Bryant \cite{bryant-remarks} and is given by
	\begin{align}\label{eq:jmap}
		j_{\g2}(\sigma) (X,Y) = *\left((X\lrcorner \g2)\wedge (Y\lrcorner \g2)\wedge \sigma\right),
	\end{align}
	where $X, Y\in \Gamma(TM)$. We used $j$ to view $\tau_3$ as either a traceless symmetric $2$-tensor or an element of $\Omega^3_{27}$ depending on the context. The decomposition of $T$ in terms of the intrinsic torsion forms is given by 
	\begin{align}\label{eq:inttorsions}
		T_{ij}=\dfrac{\tau_0}{4}g_{ij}-\frac 14 j_{\g2}(\tau_3)_{ij}-(\tau_1 \lrcorner \g2)_{ij}-\frac 12({\tau_{2}}_{ij}).
	\end{align}
	We also have the isomorphism $\Omega^1(M)\cong \Omega^2_7$ so we can view $\tau_1$ as a $1$-form or a $2$-form in $\Omega^2_7$ and explicitly $(\tau_1)_{ij}=(\tau_1)^{l}\g2_{lij}$. We also denote the $7$dimensional component of $T$ by $\Vop{T}$ with $(\Vop{T})_k=T^{ij}\g2_{ijk}$.
	
	By the expression of $d\psi$ above, a $\G2$-structure is co-closed if and only if $\tau_1=0$ and $\tau_2=0$.  Hence, the full torsion tensor of a co-closed $\G2$-structure is a symmetric $2$-tensor
	\begin{equation}
		\label{eq:torsion.coclosed}
		T=\frac{\tau_0}{4}g -\frac{1}{4}j_{\varphi}(\tau_3) 
		\in S^2.    
	\end{equation}
	
	\medskip
	
	The covariant derivative of the torsion of $\g2$ and the Riemann curvature tensor of the underlying metric are both second order in $\g2$ and they are related and we can write the expression of the Ricci tensor of the metric $g_{\g2}$ in terms of the intrinsic torsions. Since we will need to compute the Ricci tensor of our families of $\G2$-structures for looking at the flows \eqref{eq:rhfeqn} and \eqref{eq:pgflw1}, we describe Bryant's method of computing the Ricci tensor below, see \cite[eq. (4.30)]{bryant-remarks}. Define a $\G2$-equivariant pairing $Q:\Omega^3\times \Omega^3\rightarrow \Omega^3$ as follows. Let $\alpha,\ \beta \in \Omega^3$ and define
	\begin{align}\label{eq:Qmap}
		Q(\alpha, \beta) = *_{\g2}\left( (*\g2)_{ijkl}(e_j\lrcorner e_i \lrcorner *_{\g2}\alpha)\wedge (e_l\lrcorner e_k\lrcorner *\beta)\right).
	\end{align}
	Using this, we can write the Ricci tensor of the $g_{\g2}$ as
	\begin{align}\label{eq:Riccitensor}
		\Ric &= -\left(\frac 32 d^*\tau_1-\frac 38 \tau_0^2 + 15 |\tau_1|^2-\frac 14 |\tau_2|^2 + \frac 12|\tau_3|^2 \right) g \nonumber \\
		& \quad + j_{\g2}\left(-\frac 54 d(*(\tau_1\wedge \psi))-\frac 14 d\tau_2 + \frac 14*d\tau_3+\frac 52 \tau_1\wedge * (\tau_1\wedge \psi)-\frac 18\tau_0\tau_3 + \frac 14 \tau_1\wedge \tau_2 + \frac 34 *(\tau_1\wedge \tau_3)  \right. \nonumber \\
		& \qquad \qquad \left. + \frac 18 *(\tau_2\wedge \tau_2) + \frac{1}{64}Q(\tau_3, \tau_3)\right).
	\end{align}
	The formula in \eqref{eq:Riccitensor}, although looks pretty complicated, but is very useful. In particular, we will only need to use it for co-closed $\G2$-structures and hence all the terms with $\tau_1$ and $\tau_2$ vanish and we get the simpler formula
	\begin{align}\label{eq:Riccicoclosed}
		\Ric = -\left(-\frac 38 \tau_0^2  + \frac 12|\tau_3|^2 \right) g + j_{\g2}\left(  \frac 14*d\tau_3-\frac 18\tau_0\tau_3  + \frac{1}{64}Q(\tau_3, \tau_3)\right).
	\end{align}
	
	\subsection{Contact Calabi-Yau manifolds}\label{subsec:cCYmanifolds}
	We now discuss $7$-dimensional contact Calabi-Yau manifolds (cCY manifolds, henceforth). The notion of a contact Calabi-Yau manifolds as odd-dimensional manifolds having transverse Calabi-Yau geometry was originally considered by Tomassini--Vezzoni \cite{tomassini-vezzoni} and Habib--Vezzoni \cite{habib-vezzoni}. We use their approach as well as the description in \cite{lotay-saearp-saavedra, lotay-saearp}.

	\begin{definition}
		\label{def:cCY}
		A $7$-dimensional \emph{contact Calabi--Yau} (cCY) manifold is a quadruple   $(M^{7},g,\eta,\Upsilon)$  such that
		\begin{enumerate}
			\item $(M,g)$ is a $7$-dimensional Sasakian manifold with contact form $\eta$;
			
			\item $\Upsilon$ is a  nowhere vanishing closed transversal $(3,0)$-form on the distribution $\cD=\ker\eta$,  with $\omega=d\eta$ and
			
			\begin{align*}
				\Upsilon \wedge \overline{\Upsilon}= \frac{-4i}{3} \omega^3,
				\quad d\Upsilon=0.  
			\end{align*} 
		\end{enumerate}	
		We shall write:
		$$ 
		\mathrm{Re}\Upsilon:= \frac{\Upsilon+\oep}{2},\quad \mathrm{Im}\Upsilon:= \frac{\Upsilon-\oep}{2i}.
		$$  	
	\end{definition}

	We now relate the cCY geometry in 7 dimensions to $\G2$ geometry by stating the following proposition (see~\cite[Corr. 6.8]{habib-vezzoni}, \cite[Prop. 2.1]{lotay-saearp-saavedra} and \cite{lotay-saearp}). 
	\begin{proposition}\label{prop:CcYG2-structure}
		Let $(M^7,g,\eta,\Upsilon)$ be a contact Calabi-Yau 7-manifold. Then $M$ carries a $1$-parameter family of co-closed $\G2$-structures defined by 
		\begin{align}\label{eq:std.phi}
			\g2 = a\eta \wedge \omega + \mathrm{Re} \Upsilon,   
		\end{align} 
		for $a>0$, and with $\omega=d\eta$. The induced metric $g_{\g2}$ equals the metric $g$ for $a=1$. Furthermore, the dual $4$-form is given by
		\begin{align}\label{eq:std.psi}
			\psi= *_{\g2}\g2 = \frac 12 \omega^2 - a\eta \wedge \mathrm{Im} \Upsilon.   
		\end{align}    
	\end{proposition}
	
	Since $d \mathrm{Re} \Upsilon =0\ d\omega=0,\ d\mathrm{Im} \Upsilon=0$ in our case and $\omega$ and $\omega\wedge\Upsilon=0$, we see that
	\begin{align*}
		d\g2 = a \omega \wedge \omega,\ \ \ \text{and}\ \ \ d\psi =0.    
	\end{align*} and hence cCY$^{7}$ have co-closed $\G2$-structures.
	
	Since we are interested in flows of $\G2$-structures on cCY manifolds, following \cite{lotay-saearp-saavedra}, we mention the following standard set-up on a cCY manifold $M^7$ which will form the initial conditions for our flows in \Cref{sec:mainresults}. We have the following
	\begin{equation}
		\label{eq: standard setup}
		\tag{S}
		\text{\parbox{.90\textwidth}
			{$\bullet$ $(M^7,g_0,\eta_0,\Upsilon_0)$ a closed contact Calabi-Yau 7-manifold;\\
				$\bullet$ $\cD_0=\ker\eta_0$ its horizontal distribution, and $\omega_0=d\eta_0$ its transverse Kähler form;\\
				$\bullet$ for $a>0$, the $\G2$-structure $(\g2_0, \psi_0)$ are defined by \eqref{eq:std.phi} and \eqref{eq:std.psi} and hence are co-closed $\G2$-structures.
		}}
	\end{equation}
	
	\subsection{Ansatz for the flows}\label{sec:ansatz}
	
	We now consider a $1$-parameter family of $\G2$-structures on contact Calabi-Yau manifold $M^7$ with the initial  $\G2$-structure defined by \eqref{eq:std.phi} and \eqref{eq:std.psi} on a cCY setup \eqref{eq: standard setup}:
	\begin{equation}
		\label{eq: phi0 and psi0} 
		\varphi_0=a\eta_0\wedge\omega_0+{\rm{Re}}\Upsilon_0\qandq
		\psi_0=\frac{1}{2}\omega_0^2-a\eta_0\wedge{\rm{Im}}\Upsilon_0.
	\end{equation}
	
	The family of $\G2$-structures parameterized by $t\in \bR$ is given by
	\begin{equation}
		\label{eq:phit}
		\varphi_t = f_th_t^2\eta_0\wedge \omega_0+h_t^3{\rm{Re}} \Upsilon_0,
	\end{equation}
	where $f_t,h_t$ are functions depending only on time. The initial condition $\varphi_{t}\mid_{t=0}=\varphi_0$ implies 
	\begin{equation}
		\label{eq:init.conds.fh}
		f_0=a
		\qandq h_0=1.
	\end{equation}
	The  warped metric induced by $\varphi_t$ is given by
	\begin{equation}
		\label{eq: metric.t}
		g_t=f_t^2\eta_0^2+h_t^2g_{\cD_0}.
	\end{equation} 
	For the $6$-dimensional distribution $(\cD_0,\omega_0,\Upsilon_0)$,
	\begin{equation*}	
		\vol_{\cD_0}=\frac{1}{3!}\omega_0^3=\frac{i}{8}\Upsilon_0\wedge\oep_0=\frac{1}{4}{\rm{Re}} \Upsilon_0\wedge{\rm{Im}}\Upsilon_0,
	\end{equation*}
	and the volume form with respect to $\varphi_t$ is given by
	\begin{equation}\label{eq: vol.t}
		\vol_t=f_th_t^6\eta_0\wedge\vol_{\cD_0}.
	\end{equation}

	Using the above metric and orientation one can compute that 
	\begin{equation}
		\label{eq:psit}
		\psi_t =*_t\varphi_t=\frac{1}{2}h_t^4\omega_0^2-f_th_t^3\eta_0\wedge{\rm{Im}}\Upsilon_0,
	\end{equation}
	which at $t=0$ is the 4-form in \eqref{eq:std.psi} as expected.

	We now compute the intrinsic torsion forms of the above family of $\G2$-structures that defines the torsion tensor $T_t$ as in \eqref{eq:inttorsions}. This is essentially done in \cite[Lemma 3,1]{lotay-saearp-saavedra} and we are redoing the computations for completeness.
	\begin{lemma}
		For the family $\varphi_t$ in \eqref{eq:phit} the only non-zero intrinsic torsion forms are $(\tau_0)_t$ and $(\tau_3)_t$ that are given by
		\begin{align}
			\label{eq:tau0t} (\tau_0)_t&= \frac{6f_t}{7h_t^2},\\
			\label{eq:tau3t}  (\tau_3)_t&=\frac{8}{7}f_t^2\eta_0\wedge\omega_0-\frac{6}{7}f_th_t{\rm{Re}}\Upsilon_0.	
		\end{align}
	\end{lemma}
	\begin{proof}
		To compute the intrinsic torsion forms we need to compute $d\varphi_t$ and $d\psi_t$. Since $d\eta_0=\omega_0$, and  $(\omega_0,\Upsilon_0)$ is a Calabi--Yau structure the form $\omega_0$, $\Upsilon_0$ are closed and $\omega_0\wedge\Upsilon_0=0$. Thus we get
		\begin{align*}
			d\varphi_t&=f_th_t^2 d(\eta_0\wedge\omega_0)+h_t^3 d({\rm{Re}}\Upsilon_0)=f_th_t^2\omega_0^2,\\
			d\psi_t&=h_t^4(d\omega_0^2)-f_th_t^3 d(\eta_0\wedge {{\rm{Im}}}\Upsilon_0)=-f_th_t^3 \omega_0\wedge {{\rm{Im}}}\Upsilon_0=0.
		\end{align*}
		From \eqref{eq:dphi_int}, \eqref{eq:dpsi_int} this implies $(\tau_1)_t=0, (\tau_2)_t=0$. 
		\begin{align*}
			(\tau_0)_t&= \frac{1}{7}*_t(d\varphi_t \wedge\varphi_t)=\frac{f_t^2h_t^4}{7} *_t(\eta_0\wedge\omega_0^3)=\frac{6f_t^2h_t^4}{7f_th_t^6} = \frac{6f_t}{7h_t^2}.
		\end{align*}
		
		\begin{align*}
			(\tau_3)_t&= *_t(d\varphi_t-(\tau_0)_t\psi_t)=*_t\left(\frac{4}{7}f_th_t^2\omega_0^2+\frac{6f_t^2h_t}{7}\eta_0\wedge{\rm{Im}}\Upsilon_0\right)\\
			&= \frac{4}{7}f_th_t^2 *_t(\omega_0^2)++\frac{6f_t^2h_t}{7}*_t(\eta_0\wedge{\rm{Im}}\Upsilon_0)=\frac{8f_t^2}{7}\eta_0\wedge\omega_0-\frac{6f_th_t}{7} {\rm{Re}}\Upsilon_0.
		\end{align*}
	\end{proof}
	
	The intrinsic torsion forms computed above can be used to compute the torsion tensor and Ricci tensor for the family $\varphi_t$. The next proposition is again done in \cite[Prop. 3.2]{lotay-saearp-saavedra}.
	
	\begin{proposition}\label{prop:Tt_ccY}
		The torsion tensor $T_t$ for the family $\varphi_t$ in \eqref{eq:phit} is given by
		\begin{align}
			\label{eq:Tt} T_t&= -\frac{3f_t^3}{2h_t^2}\eta_0^2+\frac{f_t}{2} g_{\cD_0}.
		\end{align}  
	\end{proposition}
	\begin{proof}
		Since $d\psi_t=0$  from \eqref{eq:torsion.coclosed} $T_t$ is  a symmetric $2$-tensor given by
		\begin{align*}
			T_t&=\frac{(\tau_0)_t}{4}g_t -\frac{1}{4}j_{\varphi_t}(\tau_3)_t.
		\end{align*}
		From \eqref{eq: metric.t}, and \eqref{eq:tau0t},  
		\begin{align*}
			(\tau_0)_t g_t &= \frac{6f_t^3}{7h_t^2}\eta_0^2+\frac{6f_t}{7}g_{\cD_0}.
		\end{align*}

		For any vectors $X,Y \in \Gamma(TM)$, from \eqref{eq:jmap} 
		\begin{align*}
			j_{\varphi_t}(\tau_3)_t (X,Y)=*_t((X\lrcorner\varphi_t)\wedge (Y\lrcorner\varphi_t)\wedge (\tau_3)_t).
		\end{align*}
		Rewriting
		\begin{align*}
			(\tau_3)_t&= \frac{8f_t}{7h_t^2}\varphi_t-2f_th_t {\rm{Re}}\Upsilon_0,
		\end{align*}
		and using the linearity of $j_{\varphi_t}$ and $j_{\g2_t}(\g2_t)=6g_t$, we get that 
		\begin{align*}
			j_{\varphi_t}(\tau_3)_t&=  \frac{8f_t}{7h_t^2} j_{\varphi_t}(\varphi_t)-2f_th_t j_{\varphi_t}({\rm{Re}}\Upsilon_0)=\frac{48f_t}{7h_t^2}g_t-2f_th_t j_{\varphi_t}({\rm{Re}}\Upsilon_0).
		\end{align*}
		For $\xi_0=\eta_0^\sharp$ and $Y\in\Gamma(TM)$,
		\begin{align}
			\label{eq:jreupsilon_xi0} j_{\varphi_t}({\rm{Re}}\Upsilon_0)(\xi_0,Y)&=*_t((\xi_0\lrcorner\varphi_t)\wedge (Y\lrcorner\varphi_t)\wedge {\rm{Re}}\Upsilon_0)=f_th_t^2 *_t(\omega_0\wedge (Y\lrcorner\varphi_t)\wedge {\rm{Re}}\Upsilon_0)=0.
		\end{align}
		On $\cD_0$, for $X,Y\in \Gamma(T\cD_0)$
		\begin{align*}
			j_{\varphi_t}({\rm{Re}}\Upsilon_0)(X,Y)&=*_t(f_t^2h_t^4(X\lrcorner (\eta_0\wedge\omega_0))\wedge(Y\lrcorner (\eta_0\wedge\omega_0))\wedge {\rm{Re}}\Upsilon_0)\\
			&+*_t(2f_th_t^5(X\lrcorner (\eta_0\wedge\omega_0))\wedge(Y\lrcorner {\rm{Re}}\Upsilon_0)\wedge {\rm{Re}}\Upsilon_0)\\
			&+*_t(h_t^6(X\lrcorner {\rm{Re}}\Upsilon_0)\wedge(Y\lrcorner {\rm{Re}}\Upsilon_0)\wedge {\rm{Re}}\Upsilon_0)
		\end{align*}
		
		Observe that for $X\in \Gamma(T\cD_0)$, $X\lrcorner (\eta_0\wedge\omega_0)=-\eta_0\wedge (X\lrcorner\omega_0)$, hence the first term in the above expression vanishes. Similarly the third term in the above expression is a $7$-form on a $6$-dimensional distribution and hence vanishes. Also since $(\omega_0,\Upsilon_0)$ defines an $\rm{SU}(3)$-structure on $\cD_0$, for $X,Y\in\Gamma(T\cD_0)$
		\begin{align*}
			(X\lrcorner\omega_0)\wedge (Y\lrcorner {\rm{Re}}\Upsilon_0)\wedge {\rm{Re}}\Upsilon_0&=-2g_{\cD_0}(X,Y)\vol_{\cD_0},
		\end{align*}
		
		and hence, 
		\begin{align}
			\begin{split} \label{eq:jreupsilonD}
				j_{\varphi_t}({\rm{Re}}\Upsilon_0)(X,Y)&=-2f_th_t^5*_t(\eta_0\wedge (X\lrcorner\omega_0)\wedge (Y\lrcorner {\rm{Re}}\Upsilon_0)\wedge {\rm{Re}}\Upsilon_0)\\ &=4f_th_t^5*_t(g_{\cD_0}(X,Y)\eta_0\wedge \vol_{\cD_0})=\frac{4}{h_t}g_{\cD_0}(X,Y).
			\end{split}
		\end{align}

		Combining everything we get that 
		\begin{align}\label{eq:jtau3t}
			j_{\varphi_t}(\tau_3)_t&=\frac{48f_t^3}{7h_t^2}\eta_0^2-\frac{8f_t}{7}g_{\cD_0}.
		\end{align}
		Now the assertion follows from using the computations of $(\tau_0)_t g_t$ and $ j_{\varphi_t}(\tau_3)_t$.
	\end{proof}
	
	We now compute the Ricci tensor of $\g2_t$ explicitly using Bryant's formula \eqref{eq:Riccitensor}.
	
	\begin{proposition}\label{prop:Rict_ccY}
		The Ricci tensor $\Ric_t$ for the family $\varphi_t$ in \eqref{eq:phit} is given by
		\begin{align}
			\label{eq:Rict} \Ric_t&= \frac{3f_t^4}{2h_t^4}\eta_0^2-\frac{f_t^2}{2h_t^2} g_{\cD_0}.
		\end{align}
	\end{proposition}
	\begin{proof}
		Since $\varphi_t$ is co-closed, by \eqref{eq:Riccicoclosed} 
		\begin{align*}
			\Ric_t = -\left(-\frac 38 (\tau_0)_t^2  + \frac 12|(\tau_3)_t|^2 \right) g_t + j_{\g2_t}\left(  \frac 14*_td(\tau_3)_t-\frac 18(\tau_0)_t(\tau_3)_t  + \frac{1}{64}Q_t((\tau_3)_t, (\tau_3)_t)\right).
		\end{align*}
		The coefficient of $g_t$ in the above expression can be easily computed by using \eqref{eq:tau0t},\eqref{eq:tau3t} and is given by 
		\begin{align*}
			-\left(-\frac 38 (\tau_0)_t^2  + \frac 12|(\tau_3)_t|^2 \right)&=-\frac{309f_t^2}{98h_t^4}.
		\end{align*}
		By \eqref{eq:jreupsilon_xi0}, \eqref{eq:jreupsilonD} and using the fact that $j_{\g2_t}(\g2_t)=6g_t$ we get that 
		\begin{align}\label{eq:jeta0omega0}
			j_{\g2_t}(\eta_0\wedge\omega_0)&=\frac{6f_t}{h_t^2}\eta_0^2+\frac{2}{f_t}g_{\cD_0}.
		\end{align}
		Thus we can compute
		\begin{align*}
			j_{\g2_t}(*_td(\tau_3)_t)& =j_{\g2_t}\left(\frac{16f_t^3}{7h_t^2}\eta_0\wedge\omega_0\right) =\frac{96f_t^4}{7h_t^4}\eta_0^2+\frac{32f_t^2}{7h_t^2}g_{\cD_0}.
		\end{align*}
		We know $j_{\g2_t}(\tau_3)_t$ from \eqref{eq:jtau3t} so the only quantity left to compute is the quadratic term $Q_t((\tau_3)_t,(\tau_3)_t)$.
		
		\medskip
		
		If we denote by $\{e_i,i=1..7\}$ a local orthogonal basis of $TM$ with respect to $g_t$ then \eqref{eq:Qmap} implies
		
		\begin{align*}
			Q_t((\tau_3)_t,(\tau_3)_t)&= *_{t}\left( (\psi_t)_{ijkl}g_t^{ii}g_t^{jj}g_t^{kk}g_t^{ll}(e_j\lrcorner e_i \lrcorner *_t(\tau_3)_t)\wedge (e_l\lrcorner e_k\lrcorner *_t(\tau_3)_t)\right)\\
			&=\frac{1024f_t^3}{49h_t^2}\eta_0\wedge\omega_0+\frac{576f_t^2}{49h_t} {\rm{Re}}\Upsilon_0.
		\end{align*}
		From \eqref{eq:jeta0omega0}, \eqref{eq:jreupsilonD} we can  now easily compute $j_{\g2_t}(Q_t((\tau_3)_t,(\tau_3)_t))$ which then gives us the result.
		
	\end{proof}
	
	Again the norm squared of the torsion and $\tr T$ in the next lemma were computed in \cite[Prop. 3.3]{lotay-saearp-saavedra}. We need the expression of $T^tT$ for the Ricci-harmonic flow.
	\begin{lemma}
		For the torsion tensor $T_t$ of the family $\g2_t$ in \eqref{eq:phit},
		\begin{align}
			\label{eq:TtranposeT} T_t^tT_t&=\frac{9f_t^4}{4h_t^4}\eta_0^2+\frac{f_t^2}{4h_t^2}g_{\cD_0},\\
			\label{eq:normTsquare} |T_t|^2_{t}&=\frac{15f_t^2}{4h_t^4},\\
			\label{eq:traceT} \tr T_t&=\frac{3f_t}{2h_t^2}.
		\end{align}
	\end{lemma}
	\begin{proof}
		Since the tensor $T_t$ is symmetric from \eqref{eq:Tt} 
		\begin{align*}
			T_t^tT_t&=\sum_{i,j}(T_t)_{ia} (T_t)_{jb}g_{t}^{ab}=\frac{f_t^2}{4}\left(\frac{g_{\cD_0}}{h_t^2}\right)+\frac{9f_t^6}{4h_t^4}\left(\frac{\eta_0^2}{f_t^2}\right)=\frac{9f_t^4}{4h_t^4}\eta_0^2+\frac{f_t^2}{4h_t^2}g_{\cD_0}.
		\end{align*}
		Similarly 
		\begin{align*}
			|T_t|_t^2&= \sum_{i,j}((T_t)_{ij}g_t^{ij})^2=6\frac{f_t^2}{4h_t^4}+\frac{9f_t^6}{4h_t^4f_t^4} =\frac{15f_t^2}{4h_t^4},
		\end{align*}
		and \begin{align*}
			\tr T_t&= \sum_i (T_t)_{ii}g_t^{ii}=-\frac{3f_t^3}{2h_t^2f_t^2}+\frac{6f_t}{2h_t^2}=\frac{3f_t}{2h_t^2}.
		\end{align*}
	\end{proof}

	\section{Solutions to flows of $\G2$-structures on contact Calabi-Yau manifolds}\label{sec:mainresults}
	
	We use the expressions of various tensors which we computed in the previous section to give explicit solutions of \eqref{eq:rhfeqn}, \eqref{eq:pgflw1}, \eqref{eq:pgflw2}, and \eqref{eq:neggradflow}.
	
	\subsection{Ricci-harmonic flow}
	
	\begin{theorem}\label{thm:RHFcCY}
		Let $(M^7, \g2_0)$ be a $7$-dimensional contact Calabi-Yau manifold with the setup \eqref{eq: standard setup}. The Ricci-harmonic flow starting with $\g2_0$ is explicitly solved by the family of co-closed $\G2$-structures given by
		\begin{equation}\label{eq:rhfsol}
			\tag{Sol}
			\begin{aligned}
				\varphi_t &= f_th_t^2\eta_0\wedge \omega_0+h_t^3{\rm{Re}} \Upsilon_0,\nonumber \\
				g_t&=f_t^2\eta_0^2+h_t^2g_{\cD_0}, \nonumber \\
				\vol_t&=f_th_t^6\eta_0\wedge\vol_{\cD_0}, \nonumber \\
				\psi_t &=\frac{1}{2}h_t^4\omega_0^2-f_th_t^3\eta_0\wedge{\rm{Im}}\Upsilon_0,
			\end{aligned}
		\end{equation}
		where $h_t = (1-13a^2t)^{\frac{5}{26}},\ \ f_t=a(1-13a^2t)^{-\frac{3}{26}}$. The solution $\g2_t$ is an ancient solution of the flow with $t\in \left(-\infty, \frac{1}{13a^2}\right)$ and the flow has a finite time singularity at $t=\frac{1}{13a^2}$. If $M$ is compact, then the solution $\g2_t$ is the unique solution starting with $\g2_0$ and the singularity is a Type I singularity.  
	\end{theorem}

	\begin{proof}
		From \eqref{eq:rhfeqn} we say that a family of $\G2$-structures $\g2_t$ with initial value $\g2_0$ is a solution to the Ricci-harmonic flow if 
		\begin{align*}
			& \dfrac{\pt \g2_t}{\pt t} = \left(-\Ric_t+3 T_t^tT_t -|T_t|_t^2g_t \right) \diamond \g2_t + (\Div_t T_t)\lrcorner \psi_t.
		\end{align*}
		For the family of conical Calabi--Yau $\G2$-structures defined in \eqref{eq:phit} it was shown in \cite[Proposition 3.7]{lotay-saearp-saavedra} $\Div_tT_t=0$. The proof essentially relies on the fact that since $\xi_0$ preserves the family of $\G2$-structures, $(\del_t)_Xg_{\cD_0}=0$ for all $X\in\Gamma(TM)$, and $(\del_t)_{\xi_0}\eta_0^2=0$ so the only non-vanishing terms in $\del_tT_t$ comes for $X\in \Gamma(T\cD_0)$ from $(\del_t)_X \eta_0^2$  which is orthogonal to $\cD$. Hence for this family the flow \eqref{eq:rhfeqn} becomes
		\begin{align*}
			& \dfrac{\pt \g2_t}{\pt t} = \left(-\Ric_t+3 T_t^tT_t -|T_t|_t^2g_t \right) \diamond \g2_t.
		\end{align*}
		
		\noindent
		Using \eqref{eq:Tt},\eqref{eq:Rict},\eqref{eq:TtranposeT}, and \eqref{eq:normTsquare} one can easily compute the the $2$-tensor 
		\begin{align*}
			-\Ric_t+3 T_t^tT_t -|T_t|_t^2g_t&= \frac{3f_t^4}{2h_t^4}\eta_0^2-\frac{5f_t^2}{2h_t^2}g_{\cD_0} = \frac{4f_t^4}{h_t^4}\eta_0^2-\frac{5f_t^2}{2h_t^4}g_{t}.
		\end{align*}
		
		\noindent
		Using \eqref{eq:diadefn} one can compute that
		\begin{align}\label{eq:etadiamondphit}
			\eta_0^2 \diamond \g2_t&= \eta_0^2(\xi_0,\xi_0)(g_t)^{-1}(\xi_0,\xi_0)(\xi_0\lrcorner\g2_t)\wedge\eta_0 = \frac{f_th_t^2}{f_t^2}\eta_0\wedge\omega_0 =  \frac{h_t^2}{f_t}\eta_0\wedge\omega_0.
		\end{align}
		We now use the fact that $g_t\diamond \g2_t = 3\g2_t$ to compute 
		\begin{align}\label{eq:RHF_iphit}
			(-\Ric_t+3 T_t^tT_t -|T_t|_t^2g_t)\diamond \g2_t&=\frac{4f_t^3}{h_t^2}\eta_0\wedge\omega_0-\frac{15f_t^2}{2h_t^4}\g2_{t}=-\frac{7f_t^3}{2h_t^2}\eta_0\wedge\omega_0-\frac{15f_t^2}{2h_t}{\rm{Re}}\Upsilon_0
		\end{align}
		Taking the time derivative of \eqref{eq:phit}, we get
		\begin{align*}
			\ptt \g2_t= \ddt(f_th_t^2)\eta_0\wedge \omega_0 + \ddt(h_t^3){\rm{Re}} \Upsilon_0,
		\end{align*}
		which on using \eqref{eq:RHF_iphit} gives the following ODEs on $f_t$ and $h_t$,
		\begin{align}\label{eq:rhfODE}
			\ddt(f_th_t^2) = \frac{-7f_t^3}{2h_t^2},\ \ \ddt(h_t^3)=\frac{-15f_t^2}{2h_t},\ \ f_0=a, h_0=0.
		\end{align}
		
		Thus, we have 
		\begin{align*}
			h^2f'+2fhh'=\frac{-7f^3}{2h^2},\ h'=\frac{-5f^2}{2h^3},
		\end{align*}
		from where we get
		\begin{align*}
			f'=\frac{3f^3}{2h^4}\ \ \ \text{and}\ \ \ h'=\frac{-5f^2}{2h^3},\ \ f(0)=a,\ h(0)=1.
		\end{align*}
		We do a separation of variables to the above equation and obtain
		\begin{align*}
			\frac{df}{f}=-\frac 35 \frac{dh}{h}
		\end{align*}
		which on integration and using the initial conditions, gives 
		\begin{align*}
			f=ah^{-\frac 35}.
		\end{align*}
		We now use this expression of $f$ in the equation for $h'$ and then solve the ODE $\ddt h = -\frac{5a}{2}h^{-\frac{21}{5}}$, which again using the initial condition $h_0=1$ gives the expression for $h$ and $f$ is obtained from the relation $f=ah^{-\frac 35}$, which are finally,
		\begin{align}\label{eq:fhfinal}
			h_t = (1-13a^2t)^{\frac{5}{26}},\ \ f_t=a(1-13a^2t)^{-\frac{3}{26}}.
		\end{align}
		This proves that the expressions in \eqref{eq:rhfsol} is a solution to the Ricci-harmonic flow which is an ancient solution with $t\in \left(-\infty, \frac{1}{13a^2}\right)$.
		
		Now suppose $M$ is compact. The uniqueness of $\g2_t$ on compact $M$ follows from the short-time existence and uniqueness result in \cite[Thm. 3.7]{dwivedi-rhf}. Recall from \cite[Prop. 3.1]{lotay-saearp-saavedra} that the norm squared of the Riemann curvature tensor of $\g2_t$ is given by
		\begin{align*}
			|\Rm_t|^2_{g_t} = \frac{1}{h_t^4}|\Rm_0^{\cD_0}|^2_{g_0} + c_0 \frac{f_t^4}{h_t^8},
		\end{align*}
		where $c_0>0$ is an explicit fixed constant. Using our expressions for $f_t$ and $h_t$, we get
		\begin{align}\label{eq:Riemrhf}
			|\Rm_t|^2_{g_t}&= (1-13a^2t)^{-\frac{10}{13}}|\Rm_0^{\cD_0}|^2_{g_0} + c_0 a^4(1-13a^2t)^{-2}.
		\end{align}
		Using our expressions for $f_t, h_t$ and \eqref{eq:normTsquare}, we have
		\begin{align}\label{eq:normTrhf}
			|T_t|^4_{g_t}&= \frac{225}{16}a^4(1-13a^2t)^{-2},
		\end{align}
		and similarly using the expression for $|\del_tT_t|^2_{g_t}$ from \cite[Prop. 3.3]{lotay-saearp-saavedra}, we get
		\begin{align}\label{eq:nablaTrhf}
			|\del_tT_t|^2_{g_t} = c_0'a^4(1-13a^2t)^{-2}\ \ \ \ c_0'>0.
		\end{align}
		Using these, we can compute the quantity $\Lambda(t)$ from \eqref{eq:Lambdadefn}, which is
		\begin{align}\label{eq:Lambdarhf}
			\Lambda(t) = \left(K^2(1-13a^2t)^{-\frac{10}{13}}+c_0a^4(1-13a^2t)^{-2} \right)^{\frac 12},
		\end{align}
		where we let $K=\text{sup}_M |\Rm_0^{\cD_0}|_{g_0} $. We see that $\Lambda(t)$ blows-up as $t\rightarrow \frac{1}{13a^2}$ and hence we have a finite-time singularity by \Cref{thm:rhflte}. Moreover, $\underset{t\in [0, \frac{1}{13a^2}]}{\text{sup}}\ t\Lambda(t) <\infty$ and hence , as per \Cref{def:sing.types}, we have a Type-I singularity at $t=\frac{1}{13a^2}$. 
	\end{proof}
	
	\begin{remark}
		As mentioned before, these are the first compact examples of Ricci-harmonic flows which have a finite-time singularity as well as the first examples of Type I singularity. We also notice that the solution described in \eqref{eq:rhfsol} are ancient but not eternal as the solution fails to exist for $t\geq \frac{1}{13a^2}$. \demo  
	\end{remark}
	
	\begin{remark}
		If we take $a\rightarrow 0$ then the solution \eqref{eq:rhfsol} can be interpreted as giving degenerate \emph{eternal} solutions to the Ricci-harmonic flow given by the Calabi-Yau distribution $\cD_0$. On the contrary, if we take $a\rightarrow \infty$ then we get degenerate solution to the Ricci-harmonic flow which is only defined for non-positive times. \demo 
	\end{remark}
	
	\begin{remark}
		From the expression of the volume form in \eqref{eq:rhfsol}, we see that
		\begin{align*}
			\vol_t = a(1-13a^2t)^{\frac{27}{26}},
		\end{align*}
		and hence the volume form is \emph{decreasing} in time and hence when $M$ is compact then its volume is strictly decreasing in time with
		\begin{align*}
			\Vol(M, g_t) \rightarrow 0\ \ \ \ \ \text{as}\ \ \ \ t\rightarrow \frac{1}{13a^2}.
		\end{align*}
		This phenomenon fits exactly in the behaviour of the volume along the Ricci-harmonic flow which is always decreasing along the flow, see \cite[eq. (4.4)]{dwivedi-rhf}.  \demo
	\end{remark}
	
	\begin{remark}
		We see from the expression in \eqref{eq:RHF_iphit} that the velocity of the flow has components in both the fibre directions and also on the Calabi-Yau distribution $\cD_0$. This is different from the cases considered in \cite{lotay-saearp-saavedra}. \demo
	\end{remark}

	We know from \cite[Thm. 5.7]{dwivedi-rhf} that the condition of uniform continuity of the evolving metrics along the Ricci-harmonic flow, together with a pointwise bound on the quantity $R+4|T|^2+2\Div (\Vop{T})$, with $R$ being the scalar curvature of the metric, would lead to long-time existence results. As a result, the formation of singularities should be related to the lack of uniform
	continuity of the evolving metrics and the blow-up of $R+4|T|^2+2\Div (\Vop{T})$, so we examine the uniform continuity of the evolving metrics for our solution. We have the following result.
	
	\begin{proposition}\label{prop:unicontmetrics}
		Let $\g2_t$ in \eqref{eq:rhfsol} be the solution to \eqref{eq:rhfeqn}. Then the underlying metrics $g_t$ are uniformly continuous in $t$ on $(-\infty, \tau]$ for any $\tau < \frac{1}{13a^2}$, but it is {\bf{not}} uniformly continuous on $[\tau, \frac{1}{13a^2})$ for any $\tau$.
	\end{proposition}
	
	\begin{proof}
		The proof follows from the expression of the metrics $g_t$ given in \eqref{eq:rhfsol} and noticing that $\eta_0^2$ and $g_{\cD_0}$ are fixed tensors so any behaviour of $g_t$ depends only on $t$ and hence the uniform continuity of the metrics $g_t$ depend on the uniform continuity of the functions $(1-13a^2t)^{-3/13}$ and $(1-13a^2t)^{5/13}$ on any sub-interval $I$ of the interval $(-\infty, \frac{1}{13a^2})$. The result follows.
	\end{proof}
	
	\subsection{Ricci-like flows in \cite{panos-georgeG2Hilbert}}
	
	We now provide explicit solutions to the flows \eqref{eq:pgflw1}-\eqref{eq:pgflw2} which were introduced in \cite{panos-georgeG2Hilbert}. For the flow defined in \eqref{eq:pgflw1} a family of $\G2$-structures $\g2_t$ with initial value $\g2_0$ is a solution of the flow if is satisfies
	\begin{align*}
		\ptt \g2_t &= \left(-\Ric_t -\frac 23 (T_t\circ (\Vop{T_t}\lrcorner \g2_t))_{\text{sym}} + \tr T_t(T_t)_{\text{sym}} \right) \diamond \g2_t + \left(\Div_t T_T +\frac 13\tr T_t \Vop{T_t} -\frac 13 T_t^t(\Vop{T_t}) \right)\lrcorner \psi_T.
	\end{align*}
	
	For the family of $\G2$-structures $\g2_t$ in \eqref{eq:phit} we have that $\Div_tT_t=0$ and since $\g2_t$ is co-closed for all $t$, $\Vop T$ which is proportional to $(\tau_1)_t$ also vanishes. Moreover \eqref{eq:Tt} implies that $T_t$ is symmetric in this case and thus the family $\g2_t$ in \eqref{eq:phit} is a solution of the variational flow if
	\begin{align}\label{eq:vf1phit}
		\ptt \g2_t &= \left(-\Ric_t + \tr T_t(T_t) \right) \diamond \g2_t
	\end{align}
	
	The $2$-tensor $-\Ric_t + \tr T_t(T_t)$ can be easily computed using \eqref{eq:Rict}, \eqref{eq:Tt}, and \eqref{eq:traceT} and is given by
	\begin{align*}
		-\Ric_t + \tr T_t(T_t)&= -\frac{15f_t^4}{4h_t^4}\eta_0^2+\frac{5f_t^2}{4h_t^2}g_{\cD_0}= -\frac{5f_t^4}{h_t^4}\eta_0^2+\frac{5f_t^2}{4h_t^4}g_{t}.
	\end{align*}
	
	Now we can compute the right hand side of the expression in \eqref{eq:vf1phit} by \eqref{eq:etadiamondphit} and the fact that $g_t\diamond\g2_t=3\g2_t$. We get that 
	\begin{align}\label{eq:diamforRl}
		(-\Ric_t + \tr T_t(T_t))\diamond\g2_t&= -\frac{5f_t^3}{h_t^2}\eta_0\wedge\omega_0+\frac{15f_t^2}{4h_t^4}\g2_{t} =  -\frac{5f_t^3}{4h_t^2}\eta_0\wedge\omega_0+\frac{15f_t^2}{4h_t} {\rm{Re}}\Upsilon_0.
	\end{align}
	As a result, the ODEs for $h_t$ and $f_t$ are
	\begin{align}\label{eq:pgflw1ODE}
		\ddt (f_th_t^2)=\frac{-5f_t^3}{4h_t^2},\ \ \ \ddt(h_t^3)=\frac{15f_t^2}{4h_t},\ \ f_0=a, h_0=1.
	\end{align}
	We can solve the system of ODEs in \eqref{eq:pgflw1ODE} in exactly the same manner as in the proof of \Cref{thm:RHFcCY} and obtain
	\begin{align*}
		f_t = a\left(1+\frac{25a^2}{2}t \right)^{-\frac{3}{10}},\ \ h_t=\left(1+\frac{25a^2}{2}t \right)^{\frac{1}{10}},
	\end{align*}
	which gives the solution $\g2_t$ with $t\in \left(-\frac{2}{25a^2}, \infty \right)$. Thus, the solution of the flow \eqref{eq:pgflw1} is immortal,
	with a finite time singularity (backwards in time) at $t=-\frac{2}{25a^2}$. This discussion proves the following,
	
	\begin{theorem}\label{thm:pgflowcCY}
		Let $(M^7, \g2_0)$ be a $7$-dimensional contact Calabi-Yau manifold with the setup \eqref{eq: standard setup}. The flow in \eqref{eq:pgflw1} starting with $\g2_0$ is explicitly solved by the family of co-closed $\G2$-structures given by
		\begin{align}\label{eq:pgflw1sol}
			\begin{split}
				\varphi_t &= f_th_t^2\eta_0\wedge \omega_0+h_t^3{\rm{Re}} \Upsilon_0, \\
				\psi_t &=\frac{1}{2}h_t^4\omega_0^2-f_th_t^3\eta_0\wedge{\rm{Im}}\Upsilon_0,
			\end{split}
		\end{align}
		where $h_t = \left(1+\frac{25a^2}{2}t \right)^{\frac{1}{10}},\ \ f_t=a\left(1+\frac{25a^2}{2}t \right)^{-\frac{3}{10}}$. The solution $\g2_t$ is an immortal solution of the flow with $t\in \left(-\frac{2}{25a^2}, \infty\right)$ (but is not eternal) and the flow has a finite time singularity, which is backwards in time, at $t=-\frac{2}{25a^2}$. If, $M$ is compact, then the solution $\g2_t$ is the unique solution starting with $\g2_0$.  
	\end{theorem}
	The uniqueness of the solutions on compact manifolds follows from \cite[Thm. 6.76]{dgk-flows}.\qed
	
	\begin{remark}
		If we take $a\rightarrow 0$ then the solution \eqref{eq:pgflw1sol} can be interpreted as giving degenerate \emph{eternal} solutions to flow \eqref{eq:pgflw1} given by the Calabi-Yau distribution $\cD_0$. On the contrary, if we take $a\rightarrow \infty$ then we get degenerate solution to the flow which is only defined for non-negative times. \demo 
	\end{remark}
	
	\begin{remark}
		From the expression of the volume form in \eqref{eq:rhfsol}, we see that for the flow \eqref{eq:pgflw1},
		\begin{align*}
			\vol_t = a\left(1+\frac{25a^2}{2}t\right)^{\frac{3}{10}},
		\end{align*}
		and hence the volume form is \emph{increasing} in time and hence when $M$ is compact then its volume is strictly increasing in time with
		\begin{align*}
			\Vol(M, g_t) \rightarrow \infty\ \ \ \ \ \text{as}\ \ \ \ t\rightarrow \infty.
		\end{align*}
		\demo
	\end{remark}
	
	\medskip
	
	\begin{remark}
		We see from the expression in \eqref{eq:diamforRl} that the velocity of the flow has components in both the fibre directions and also on the Calabi-Yau distribution $\cD_0$. Again, this is different from the cases considered in \cite{lotay-saearp-saavedra}. \demo
	\end{remark}
	
	Even though, there is no analogue of \Cref{thm:rhflte} for the flow in \eqref{eq:pgflw1} yet, based on the analysis of various other flows of $\G2$-structures, it is reasonable to expect that the quantity $\Lambda(t)$ from \eqref{eq:Lambdadefn} controls the maximal existence time of the flow in \eqref{eq:pgflw1}. So if we assume this aspect, then on compact $7$-dimensional cCY manifolds we get the following result for the solution $\g2_t$ in \eqref{eq:pgflw1sol}.
	
	\begin{proposition}
		In the setup \eqref{eq: standard setup}, if $M$ is compact, then the solution $\g2_t$ in \eqref{eq:pgflw1sol} has a Type IIb infinite time singularity unless $g_{\cD_0}$ is flat, in which case the singularity at infinite time is of Type III.
	\end{proposition}
	\begin{proof}
		The result follows from computing the expression for $\Lambda(t)$ by using $f_t$, $h_t$ and \eqref{eq:Riemrhf}, \eqref{eq:normTrhf} and \eqref{eq:nablaTrhf}. In fact, we have
		\begin{align}
			\Lambda(t) = \left(K^2(1+\frac{25a^2}{2}t)^{-\frac{2}{5}}+c_0a^4(1+\frac{25a^2}{2}t)^{-2} \right)^{\frac 12}
		\end{align}
		where again we let $K=\text{sup}_M |\Rm_0^{\cD_0}|_{g_0} $. The assertions on the type of singularity now follows from \Cref{def:sing.types}.
	\end{proof}
	
	\begin{remark}
		If a result like \Cref{thm:rhflte} is true for the Ricci-like flow \eqref{eq:pgflw1} then $\g2_t$ in \eqref{eq:pgflw1sol} will be the first compact examples of infinite time singularity of the flow which are Type IIb and Type III.
	\end{remark}

	For the flow in \eqref{eq:pgflw2}, we can compute
	\begin{align*}
		-\Ric_t + \tr T_t(T_t) +\frac{|T_t|^2}{3}g_t = -\frac{5f_t^4}{h_t^4}\eta_0^2+\frac{5f_t^2}{2h_t^4}g_{t}.
	\end{align*}
	Thus using \eqref{eq:etadiamondphit} and $g_t\diamond\g2_t=3\g2_t$ we get that
	\begin{align*}
		\left(-\Ric_t + \tr T_t(T_t) +\frac{|T_t|^2}{3}g_t\right) \diamond \g2_t&= \frac{5f_t^3}{2h_t^2}\eta_0\wedge\omega_0+\frac{15f_t^2}{2h_t} {\rm{Re}}\Upsilon_0,
	\end{align*}
	and hence the corresponding ODEs, are 
	\begin{align*}
		\ddt(f_th_t^2)=\frac{5f_t^3}{2h_t^2},\ \ \ddt(h_t^3)=\frac{15f_t^2}{2h_t},\ \ f_0=a, h_0=1,
	\end{align*}
	which gives the solutions
	\begin{align*}
		f_t = a(1+15a^2t)^{-\frac 16},\ \ h_t=(1+15a^2t)^{\frac 16},
	\end{align*}
	and so we again get immortal solutions of the flow. 
	
\medskip

We now look at the behaviour of the $\G2$-Einstein-Hilbert functional along the flow $\g2(t)$, since the flows \eqref{eq:pgflw1} and \eqref{eq:pgflw2} are obtained by looking at the variation of the functional\footnote{We thank Panagiotis Gianniotis for suggesting this.}. The $\G2$-Einstein-Hilbert functional was defined in \cite{panos-georgeG2Hilbert} and is given by
\begin{align*}
\cF(\g2) = \int_M \left(\frac 16 R-\frac 13 |T|^2 - \frac 16 (\tr T)^2 \right)\vol,
\end{align*} 
where $R$ is the scalar curvature of the underlying metric $g_{\g2}$. It can also be re-written as
\begin{align*}
\cF(\g2) = \int_M \left(-\frac 12 |T|^2+\frac 16 |\Vop{T}|^2\right)\vol.
\end{align*}
Since along the family $\g2_t$ in \eqref{eq:phit}, $\Vop_t(T_t)=0$, the functional reduces to
\begin{align}\label{eq:G2EH}
\cF(\g2_t) = \int_M-\frac 12 |T_t|^2 \vol_t. 
\end{align}
Using \eqref{eq:normTsquare} and the expression of $\vol_t$ from \eqref{eq:rhfsol}, we see that
\begin{align*}
\cF(\g2_t)&= -\frac{15}{8}\int_M a^{2}\left(1+\frac{25a^2}{2}t\right)^{-\frac 75}a\left(1+\frac{25a^2}{2}t\right)^{\frac{3}{10}} \eta_0\wedge \vol_{\cD_0} \\
&= -\frac{1}{2}\left(1+\frac{25}{a^2}t\right)^{-\frac{11}{10}} \cF(\g2_0).
\end{align*}
Thus, the $\G2$-Einstein-Hilbert functional goes to $0$ as $t\rightarrow \infty$ and tends to $\infty$ as $t\rightarrow -\frac{2}{25a^2}$.

\subsection{Negative gradient flow of an energy functional}\label{subsec:ngf}

We now analyze the negative gradient flow of the energy functional which is $L^2$-norm of the torsion on contact Calabi-Yau manifolds with our ansatz \eqref{eq:phit}. Recall that a natural energy functional on the space $\Omega^3_{+}(M)$ of $\G2$-structures is 
\begin{align*}
E(\g2) = \frac 12\int_M |T|^2\vol.
\end{align*}
It is known, for instance see \cite[Corr. 5.10]{dgk-flows}, that if $\ptt \g2=h\diamond \g2 +X\lrcorner \psi$ on a compact manifold $M$, then the variation of $E$ is given by
\begin{align*}
\frac 12\ptt\int_M |T|^2\vol = \int_M \left\langle \Ric +\frac 12\cL_{\Vop{T}}g + \frac 12|T|^2g-(\tr T)T_{\text{sym}}+T^2_{\text{sym}}-TT^t-(T(PT))_{\text{sym}}, h\right \rangle \vol + \int_M \left\langle \Div T, X\right\rangle \vol,
\end{align*}
where $(PT)_{ij}= T_{ab}\psi_{abij}$. Thus, the negative gradient flow of the energy functional $E$ is the following flow of $\G2$-structures, 
\begin{align}\label{eq:neggradflow}
\ptt \g2 = \left(-\Ric -\frac 12\cL_{\Vop{T}}g - \frac 12|T|^2g+(\tr T)T_{\text{sym}}-T^2_{\text{sym}}+TT^t-(T(PT))_{\text{sym}}\right) + \Div T\lrcorner \psi.
\end{align}

Since we are interested in the behaviour of the flow along the family $\g2_t$ in \eqref{eq:phit}, along which the $\G2$-structures are co-closed hence $T$ is symmetric and $\tr T$ is constant and thus,
\begin{align*}
\Div T_t=0, \ \Vop_t{T_t}=0, \ (\tr T)T_{\text{sym}}=(\tr T)T, \ -T^2_{\text{sym}}=-T^2, \ PT=0, \ \text{and}\  TT^t=T^2.
\end{align*}
Consequently, the flow in \eqref{eq:neggradflow}, for the family $\g2_t$ in \eqref{eq:phit}, reduces to
\begin{align}\label{eq:neggradphit}
\ptt \g2_t = \left(-\Ric -\frac 12|T|^2g+(\tr T)T\right)\diamond \g2_t.
\end{align}

Proceeding in the same way as in the previous two subsections, we get
\begin{align*}
\left(-\Ric_t -\frac 12|T_t|^2g_t+(\tr T_t)T_t\right)\diamond_t \g2_t = -\frac{27f_t^3}{8h_t^2}\eta_0\wedge \omega_0 - \frac{27f_t^2}{8h_t}\rm{Re} \Upsilon_0
\end{align*}
and as a result, the ODEs for the functions $f_t$ and $h_t$ are
\begin{align}\label{eq:odeneggrad}
\ddt(f_th_t^2)=-\frac{27f_t^3}{8h_t^2},\ \ \ \ddt(h_t^3)=-\frac{27f_t^2}{8h_t},\ \ \ f_0=a,\ h_0=1.
\end{align}
We can again solve the ODEs using the separation of variables as before and as a result, the solution for \eqref{eq:odeneggrad} are given by
\begin{align*}
f_t=a\left(1-\frac{9a^2}{4}t\right)^{\frac 12}, \ \ \ \ h_t=\left(1-\frac{9a^2}{4}t\right)^{\frac 12}.
\end{align*}
The solution $\g2_t$ exists for all $t\in (-\infty, \frac{4a^2}{9}]$ and has a finite time singularity at $t=\frac{4a^2}{9}$. This discussion proves the following 

\begin{theorem}\label{thm:ngflowcCY}
	Let $(M^7, \g2_0)$ be a $7$-dimensional contact Calabi-Yau manifold with the setup \eqref{eq: standard setup}. The flow in negative gradient flow of the energy functional \eqref{eq:neggradflow} starting with $\g2_0$ is explicitly solved by the family of co-closed $\G2$-structures given by
	\begin{align}\label{eq:ngflw1sol}
		\begin{split}
			\varphi_t &= f_th_t^2\eta_0\wedge \omega_0+h_t^3{\rm{Re}} \Upsilon_0, \\
			\psi_t &=\frac{1}{2}h_t^4\omega_0^2-f_th_t^3\eta_0\wedge{\rm{Im}}\Upsilon_0,
		\end{split}
	\end{align}
	where $h_t = \left(1-\frac{9a^2}{4}t \right)^{\frac{1}{2}},\ \ f_t=a\left(1-\frac{9a^2}{4}t \right)^{\frac{1}{2}}$. The solution $\g2_t$ is an ancient solution of the flow with $t\in \left(-\infty, \frac{4}{9a^2}\right)$ (but is not eternal) and the flow has a finite time singularity at $t=\frac{4}{9a^2}$. If, $M$ is compact, then the solution $\g2_t$ is the unique solution starting with $\g2_0$.  
\end{theorem}
The uniqueness of the solutions on compact manifolds follows from the result in \cite{weiss-witt} as well as \cite[Thm. 6.76]{dgk-flows}.\qed

Just like the Ricci-like flows in the previous section, there is no analogue of \Cref{thm:rhflte} for the flow in \eqref{eq:neggradflow} yet, since the flow again is a Ricci-like flow, it is reasonable to expect that the quantity $\Lambda(t)$ from \eqref{eq:Lambdadefn} controls the maximal existence time of the flow in \eqref{eq:neggradflow}. So if we assume this aspect, then on compact $7$-dimensional cCY manifolds we get the following result for the solution $\g2_t$ in \eqref{eq:ngflw1sol}.

\begin{proposition}
	In the setup \eqref{eq: standard setup}, if $M$ is compact, then the solution $\g2_t$ in \eqref{eq:ngflw1sol} has a finite time Type I singularity at $t=\frac{4}{9a^2}$.
\end{proposition}
\begin{proof}
	The result follows from computing the expression for $\Lambda(t)$ by using $f_t$, $h_t$ and \eqref{eq:Riemrhf}, \eqref{eq:normTrhf} and \eqref{eq:nablaTrhf}. In fact, we have
	\begin{align}
		\Lambda(t) = \left(K^2(1-\frac{9a^2}{4}t)^{-\frac{2}{5}}+c_0a^4(1-\frac{9a^2}{4}t)^{-2} \right)^{\frac 12}
	\end{align}
	where again we let $K=\text{sup}_M |\Rm_0^{\cD_0}|_{g_0} $. The assertions on the type of singularity now follows from \Cref{def:sing.types}.
\end{proof}

\begin{remark}
	If a result like \Cref{thm:rhflte} is true for the negative gradient flow of the energy functional \eqref{eq:neggradflow} then $\g2_t$ in \eqref{eq:ngflw1sol} will be the first compact examples of fine time singularity of the flow which are Type I.
\end{remark}

	\section{Solutions of the flows on the 7-dimensional Heisenberg group}\label{sec:solonheisenberg}
	In this section, we consider the solution of the Ricci-harmonic flow and the flows in \eqref{eq:pgflw1}, \eqref{eq:pgflw2}, and \eqref{eq:neggradflow} on the seven dimensional Heisenberg group $H$ which is a simply connected nilpotent Lie group. The reader is referred to the work \cite{bff-coflow} of Bagaglini--Fernández--Fino for more details on $\G2$-structures on $H$. We now describe a family of $\G2$-structures on the 7-dimensional Heisenberg group $H$. We recall that 
	
	\begin{definition}
		The Heisenberg Lie algebra $\frak{g}$ is a real $(2n+1)$-dimensional Lie algebra with two $n$-dimensional abelian sub-algebras $\eta_1,\eta_2$ and 1-dimensional subalgebra $\xi$ such that \begin{align*}
			\frak{h}&=\eta_1\oplus\eta_2\oplus\xi,\\
			[\eta_i,\eta_i]&=[\eta_i,\xi]=0, \ \ \ i=1,2.
		\end{align*}
		
		\noindent
		Moreover there exists basis $\eta_1={\rm{Span}}\{X_1,\ldots,X_n\}$, $\eta_2={\rm{Span}}\{Y_1,\ldots,Y_n\}$, and $\xi={\rm{Span}}\{Z\}$ for which
		\begin{align*}
			[X_i,Y_j]&=-\delta_{ij} Z \ \ \ \ \text{for all} \ 1\leq i,j\leq n.
		\end{align*}
	\end{definition}

	In dimension $7$ we choose a basis $e_1,\ldots,e_7$ of the Lie algebra $\frak{h}$ so the structure equations are given by 
	\[[e_1,e_2]=-e_7, [e_3,e_4]=-e_7 ,[e_5,e_6]=-e_7.\]If we denote by $\{e^i,i=1,\ldots,6, e^7\}$ the dual basis of $\frak{h}^*$, we can write
	\begin{align*}
		\frak{h}&= \left(0,0,0,0,0,0,e^{12}+e^{34}+e^{56}\right),
	\end{align*}
	which implies that 
	\begin{align}\label{eqn:diff_h7}
		de^i=0, i=1,\ldots,6, \ \ \ de^7=e^{12}+e^{34}+e^{56},
	\end{align}
	where $e^{ij}=e^i\wedge e^j$.
	
	\subsection{Ansatz for the flows}
	In the basis leading to \eqref{eqn:diff_h7}, we can define a left-invariant $\G2$-structure on $H$ by
	\begin{align}\label{eq:3formheisenberg}
		\g2_0&= e^{127}+e^{347}+e^{567}+ e^{135}-e^{146}-e^{236}-e^{245}.
	\end{align}
	The above $\G2$-structure defines the metric and orientation
	\begin{align*}
		g_0&= \sum_{i=1}^7 (e^i)^2,\\
		\vol_0&= e^1\wedge\ldots\wedge e^7.
	\end{align*}
	The $4$-form $\psi_0=*_0\varphi_0$ is given by
	\begin{align*}
		\psi_0&= e^{1234}+e^{3456}+e^{1256}+e^{1367}+e^{1457}+e^{2357}-e^{2467}.
	\end{align*}
	We can now define a family of left invariant $\G2$-structures parameterized by $t\in\mathbb{R}$ given by
	\begin{align}\label{eqn:phit_h7}
		\g2_t&= f_t(a_t^2e^{127}+b_t^2e^{347}+c_t^2e^{567})+a_t b_t c_t(e^{135}-e^{146}-e^{236}-e^{245}).
	\end{align}
	such that $f(0)=1, \ a(0)=b(0)=c(0)=1$.
	For this family, the metric and the orientation are given by
	\begin{align*}
		g_t&= a_t^2( (e^1)^2+(e^2)^2)+b_t^2( (e^3)^2+(e^4)^2)+c_t^2( (e^5)^2+(e^6)^2) + f_t^2 (e^7)^2,\\
		\vol_t&=f_ta_t^2b_t^2c_t^2 \vol_0.
	\end{align*}
	The $4$-form is
	\begin{align*}
		\psi_t&=*_t\g2_t=(a_t^2b_t^2 e^{1234}+b_t^2c_t^2 e^{3456}+a_t^2 c_t^2 e^{1256})+f_ta_tb_tc_t (e^{1367}+e^{1457}+e^{2357}-e^{2467}).
	\end{align*}
	We immediately observe that for $t=0$ the initial $\G2$-structure is given by $\g2_0$. 
	Using the differential relations in \eqref{eqn:diff_h7} we can compute the intrinsic torsion forms for this family by computing 
	\begin{align*}
		d\g2_t&= f_t((a_t^2+b_t^2)e^{1234}+(b_t^2+c_t^2)e^{3456}+(a_t^2+c_t^2)e^{1256}),\\
		d\psi_t&= 0.
	\end{align*}
	
	Thus this defines a family of co-closed $\G2$-structures and the only non-zero torsion forms from \eqref{eq:inttorsions} are $(\tau_0)_t,(\tau_3)_t$ which are given by
	\begin{align*}
		(\tau_0)_t&= \frac{2f_t (a_t^2b_t^2+a_t^2c_t^2+b_t^2c_t^2)}{7a_t^2b_t^2c_t^2},\\
		(\tau_3)_t&= \frac{f_t^2(5a_t^2(b_t^2+c_t^2)-2b_t^2c_t^2}{7a_t^2b_t^2}e^{127}+ \frac{f_t^2(5b_t^2(a_t^2+c_t^2)-2a_t^2c_t^2}{7a_t^2c_t^2}e^{347}+\frac{f_t^2(5c_t^2(a_t^2+b_t^2)-2a_t^2b_t^2}{7a_t^2c_t^2} e^{567})\\
		& \quad -\frac{2f_t(a_t^2b_t^2+a_t^2c_t^2+b_t^2c_t^2)}{7a_tb_tc_t}(e^{135}-e^{146}-e^{236}-e^{245}).
	\end{align*}
	Using \eqref{eq:torsion.coclosed} we can then compute the torsion tensor $T_t$ for $\g2_t$. 
	\begin{proposition}
		For the family of $\G2$-structures in \eqref{eqn:phit_h7} the torsion tensor $T_t$ is given by
		\begin{align}\label{eqn:Tt_h7}
			T_t&= \frac{f_t}{2} \sum_{i=1}^6 (e^i)^2 -\frac{f_t^3(a_t^2b_t^2+a_t^2c_t^2+b_t^2c_t^2)}{2a_t^2b_t^2c_t^2} (e^7)^2.
		\end{align}
		\qed
	\end{proposition}
	
	\noindent
	Using the intrinsic torsion forms the Ricci tensor of $\g2_t$ can also be explicitly computed using Bryant's formula \eqref{eq:Riccitensor}.
	
	\begin{proposition}
		The Ricci tensor $\Ric_t$ for the family $\varphi_t$ in \eqref{eqn:phit_h7} is given by
		\begin{align}
			\label{eq:Rict_h7} \Ric_t&=  -\frac{f_t^2}{2a_t^2} ( (e^1)^2+(e^2)^2)- \frac{f_t^2}{2b_t^2} ( (e^3)^2+(e^4)^2)-\frac{f_t^2}{2c_t^2} ( (e^5)^2+(e^6)^2) + \frac{f_t^4(a_t^4b_t^4+a_t^4c_t^4+b_t^4c_t^4)}{2a_t^4b_t^4c_t^4} (e^7)^2
		\end{align}
	\end{proposition}
	The proofs for the above propositions  are similar to the proof of Proposition~\ref{prop:Tt_ccY} and Proposition~\ref{prop:Rict_ccY} respectively and are hence omitted. We also record the following identities for later use.
	
	\begin{lemma}
		For the torsion tensor $T_t$ of the family $\g2_t$ in \eqref{eqn:phit_h7},
		\begin{align}
			\label{eq:TtranposeT_h7} T_t^tT_t&= \frac{f_t^2}{4a_t^2} ( (e^1)^2+(e^2)^2)+\frac{f_t^2}{4b_t^2} ( (e^3)^2+(e^4)^2)+\frac{f_t^2}{4c_t^2} ( (e^5)^2+(e^6)^2) + \frac{f_t^4(a_t^2b_t^2+a_t^2c_t^2+b_t^2c_t^2)^2}{4a_t^4b_t^4c_t^4} (e^7)^2,\\
			\label{eq:normTsquare_h7} |T_t|^2_{t}&=\frac{f_t^2 (2a_t^4b_t^4+2a_t^4c_t^4+2b_t^4c_t^4+(a_t^2b_t^2+a_t^2c_t^2+b_t^2c_t^2)^2)}{4a_t^4b_t^4c_t^4},\\
			\label{eq:traceT_h7} \tr T_t&=\frac{f_t(a_t^2b_t^2+a_t^2c_t^2+b_t^2c_t^2)}{2a_t^2b_t^2c_t^2}.
		\end{align}
	\end{lemma}
	
	For the family of $\G2$-structures defined in \eqref{eqn:phit_h7} it was shown in \cite{bff-coflow} $\Div_tT_t=0$ as it only depends on the derivative of $(\tau_0)_t$. Hence for this family the flow \eqref{eq:rhfeqn} becomes
	\begin{align*}
		& \dfrac{\pt \g2_t}{\pt t} = \left(-\Ric_t+3 T_t^tT_t -|T_t|_t^2g_t \right) \diamond \g2_t.
	\end{align*}
	
	\noindent
	Using \eqref{eqn:Tt_h7},\eqref{eq:Rict_h7},\eqref{eq:TtranposeT_h7}, and \eqref{eq:normTsquare_h7} one can easily compute the the $2$-tensor 
	\begin{align*}
		-\Ric_t+3 T_t^tT_t -|T_t|_t^2g_t&=  -\frac{f_t^2(3a_t^4(b_t^4+c_t^4)-2b_t^4c_t^4+2(a_tb_tc_t)^2(a_t^2+b_t^2+c_t^2))}{4a_t^2b_t^4c_t^4} \ ( (e^1)^2+(e^2)^2)\\
		& \qquad -\frac{f_t^2(3b_t^4(a_t^4+c_t^4)-2a_t^4c_t^4+2(a_tb_tc_t)^2(a_t^2+b_t^2+c_t^2))}{4a_t^4b_t^2c_t^4} \  ( (e^3)^2+(e^4)^2)\\
		&\qquad -\frac{f_t^2(3c_t^4(a_t^4+b_t^4)-2a_t^4b_t^4+2(a_tb_tc_t)^2(a_t^2+b_t^2+c_t^2))}{4a_t^4b_t^4c_t^2} \ ( (e^5)^2+(e^6)^2) \\
		&\qquad + \frac{f_t^4(2(a_tb_tc_t)^2(a_t^2+b_t^2+c_t^2)-a_t^4b_t^4+-a_t^4c_t^4-b_t^4c_t^4)}{2a_t^4b_t^4c_t^4} \ (e^7)^2.
	\end{align*}
	
	For the above tensor we can now compute 
	\begin{align*}
		i_{\g2_t}(-\Ric_t+3 T_t^tT_t -|T_t|_t^2g_t) &= -\frac{f_t^3(4a_t^4(b_t^4+c_t^4)-b_t^4c_t^4)}{2a_t^2b_t^4c_t^4}e^{127}-\frac{f_t^3(4b_t^4(a_t^4+c_t^4)-a_t^4c_t^4)}{2a_t^4b_t^2c_t^4}e^{347} \nonumber \\
		& \qquad -\frac{f_t^3(4c_t^4(a_t^4+b_t^4)-a_t^4b_t^4)}{2a_t^4b_t^4c_t^2}e^{567}\\
		& \qquad -\frac{f_t^2(3(a_tb_tc_t)^2(a_t^2+b_t^2+c_t^2)+2(a_t^4b_t^4+a_t^4c_t^4+b_t^4c_t^4))}{2a_t^3b_t^3c_t^3} \ (e^{135}-e^{146}-e^{236}-e^{245})
	\end{align*}
	
	The system of ODEs for the family \eqref{eqn:phit_h7} to be a solution for \eqref{eq:rhfeqn} is given by
	\begin{align*}
		\frac{d(f_ta_t^2)}{dt}&= -\frac{f_t^3(4a_t^4(b_t^4+c_t^4)-b_t^4c_t^4)}{2a_t^2b_t^4c_t^4},\\
		\frac{d(f_tb_t^2)}{dt}&=-\frac{f_t^3(4b_t^4(a_t^4+c_t^4)-a_t^4c_t^4)}{2a_t^4b_t^2c_t^4},\\
		\frac{d(f_tc_t^2)}{dt}&=-\frac{f_t^3(4c_t^4(a_t^4+b_t^4)-a_t^4b_t^4)}{2a_t^4b_t^4c_t^2},\\
		\frac{d(a_tb_tc_t)}{dt}&=-\frac{f_t^2(3(a_tb_tc_t)^2(a_t^2+b_t^2+c_t^2)+2(a_t^4b_t^4+a_t^4c_t^4+b_t^4c_t^4))}{2a_t^3b_t^3c_t^3}.
	\end{align*}
	The previous discussion can be summarized in the following, very general theorem on the Heisenberg group.
	
	\begin{theorem}\label{thm:RHFheisenberggeneral}
		Consider the $7$-dimensional Heisenberg group with Lie algebra and structure constants given in \eqref{eqn:diff_h7}. Then for the initial $\G2$-structure $\g2_0$ in \eqref{eq:3formheisenberg}, the family $\g2_t$ given in \eqref{eqn:phit_h7} is a solution to the Ricci-harmonic flow if and only if the functions $f_t,\ a_t,\ b_t,\ c_t$ with initial conditions $f(0)=a(0)=b(0)=c(0)=1$ satisfy the following system of ODE,
		\begin{align}\label{eq:heisenbergODERHF}
			\begin{split}
				\frac{d(f_ta_t^2)}{dt}&= -\frac{f_t^3(4a_t^4(b_t^4+c_t^4)-b_t^4c_t^4)}{2a_t^2b_t^4c_t^4},\\
				\frac{d(f_tb_t^2)}{dt}&=-\frac{f_t^3(4b_t^4(a_t^4+c_t^4)-a_t^4c_t^4)}{2a_t^4b_t^2c_t^4},\\
				\frac{d(f_tc_t^2)}{dt}&=-\frac{f_t^3(4c_t^4(a_t^4+b_t^4)-a_t^4b_t^4)}{2a_t^4b_t^4c_t^2},\\
				\frac{d(a_tb_tc_t)}{dt}&=-\frac{f_t^2(3(a_tb_tc_t)^2(a_t^2+b_t^2+c_t^2)+2(a_t^4b_t^4+a_t^4c_t^4+b_t^4c_t^4))}{2a_t^3b_t^3c_t^3}.
			\end{split} 
		\end{align}
	\end{theorem}
	The equation \eqref{eq:heisenbergODERHF} is very general and solutions to the ordinary differential equations will give explicit solutions to the Ricci-harmonic flow. 
	
	Now we specialize to the case when $a(t)=b(t)=c(t)=h(t)$, with $h(0)=1$. We can solve the ODEs in \eqref{eq:heisenbergODERHF} in this case and we get the following result which is (as expected) similar to \Cref{thm:RHFcCY} and so we omit the proof.
	
	\begin{theorem}\label{thm:RHFheisenberg}
		Let $(M^7, \g2_0)$ with $\g2_0$ in \eqref{eq:3formheisenberg} be the $7$-dimensional Heisenberg group. The Ricci-harmonic flow starting with $\g2_0$ is explicitly solved by the family of co-closed $\G2$-structures given by
		\begin{equation}\label{eq:rhfheisol}
			\begin{aligned}
				\varphi_t &= f_th_t^2(e^{127}+e^{347}+e^{567})+h_t^3(e^{135}-e^{146}-e^{236}-e^{245})
			\end{aligned}
		\end{equation}
		where $h_t = (1-13t)^{\frac{5}{26}},\ \ f_t=a(1-13t)^{-\frac{3}{26}}$. The solution $\g2_t$ is an ancient solution of the flow with $t\in \left(-\infty, \frac{1}{13}\right)$ and the flow has a finite time singularity at $t=\frac{1}{13}$. If $M$ is compact, then the solution $\g2_t$ is the unique solution starting with $\g2_0$ and the singularity is a Type I singularity.  
		\qed 
	\end{theorem}
The behaviour of the quantity $\Lambda(t)$, the type of singularity of the solutions as well as the range for the uniform continuity of the evolving metrics follow in the same way as in the cCY case.

\medskip
	
We now derive the system of ODEs for the functions $f(t),\ a(t),\ b(t)$ and $c(t)$ in \eqref{eqn:phit_h7} which they need to satisfy in order to get solutions of the Ricci-like flows in \cite{panos-georgeG2Hilbert}.

For \eqref{eq:pgflw1}, since $(\tau_1)_t=0$ for the family of $\G2$-structures in \eqref{eqn:phit_h7}, the equation for \eqref{eq:pgflw1} becomes
	\begin{align*}
		\ptt \g2_t&=  -\Ric_t + \tr T_t(T_t).
	\end{align*}
	Using \eqref{eq:traceT_h7}, \eqref{eq:TtranposeT_h7}, and \eqref{eq:Rict_h7} we can compute that for $\g2_t$ in \eqref{eqn:phit_h7}
	\begin{align*}
		i_{\g2_t}( -\Ric_t + \tr T_t(T_t)) &=-\frac{f_t^3(3a_t^4(b_t^4+c_t^4)-3b_t^4c_t^4+2a_t^4b_t^2c_t^2)}{4a_t^2b_t^4c_t^4} e^{127}-\frac{f_t^3(3b_t^4(a_t^4+c_t^4)-3a_t^4c_t^4+2a_t^2b_t^4c_t^2)}{4a_t^4b_t^2c_t^4}e^{347} \nonumber \\
		& \qquad -\frac{f_t^3(3c_t^4(a_t^4+b_t^4)-3a_t^4b_t^4+2a_t^2b_t^2c_t^4)}{4a_t^4b_t^4c_t^2}e^{567}\\
		& \qquad +\frac{f_t^2(2(a_tb_tc_t)^2(a_t^2+b_t^2+c_t^2)+3(a_t^4b_t^4+a_t^4c_t^4+b_t^4c_t^4))}{4a_t^3b_t^3c_t^3} \ (e^{135}-e^{146}-e^{236}-e^{245}).
	\end{align*}
	
\noindent	
So the system of ODEs we get for the family \eqref{eqn:phit_h7} to be a solution for \eqref{eq:pgflw1} is given by
	\begin{align}\label{eq:pgflw1heigen}
\begin{split}		
		\frac{d(f_ta_t^2)}{dt}&= \frac{f_t^3(3a_t^4(b_t^4+c_t^4)-3b_t^4c_t^4+2a_t^4b_t^2c_t^2)}{4a_t^2b_t^4c_t^4},\\
		\frac{d(f_tb_t^2)}{dt}&=-\frac{f_t^3(3b_t^4(a_t^4+c_t^4)-3a_t^4c_t^4+2a_t^2b_t^4c_t^2)}{4a_t^4b_t^2c_t^4},\\
		\frac{d(f_tc_t^2)}{dt}&=-\frac{f_t^3(3c_t^4(a_t^4+b_t^4)-3a_t^4b_t^4+2a_t^2b_t^2c_t^4)}{4a_t^4b_t^4c_t^2},\\
		\frac{d(a_tb_tc_t)}{dt}&=\frac{f_t^2(2(a_tb_tc_t)^2(a_t^2+b_t^2+c_t^2)+3(a_t^4b_t^4+a_t^4c_t^4+b_t^4c_t^4))}{4a_t^3b_t^3c_t^3}.
	\end{split}	
	\end{align}

Similarly, for the flow defined by \eqref{eq:pgflw2} for the co-closed family of $\G2$-structures in \eqref{eqn:phit_h7}, the equation for \eqref{eq:pgflw2} becomes
	\begin{align*}
		\ptt \g2_t&=  -\Ric_t + \tr T_t(T_t)+\frac{|T_t|^2}{3}g_t.
	\end{align*}
Again we can use \eqref{eq:traceT_h7}, \eqref{eq:TtranposeT_h7}, \eqref{eq:normTsquare_h7}, and \eqref{eq:Rict_h7} to compute for $\g2_t$ in \eqref{eqn:phit_h7}
\begin{align*}
		i_{\g2_t}( -\Ric_t + \tr T_t(T_t)++\frac{|T_t|^2}{3}g_t) &=\frac{f_t^3(a_t^2(b_t^2+c_t^2)+3b_t^2c_t^2)}{2a_t^2b_t^2c_t^2} e^{127}+\frac{f_t^3(b_t^2(a_t^2+c_t^2)+3a_t^2c_t^2)}{2a_t^2b_t^2c_t^2} e^{347} \nonumber \\
		& \quad +\frac{f_t^3(c_t^2(a_t^2+b_t^2)+3a_t^2b_t^2)}{2a_t^2b_t^2c_t^2}e^{567}\\
		& \quad +\frac{f_t^2(2(a_tb_tc_t)^2(a_t^2+b_t^2+c_t^2)+3(a_t^4b_t^4+a_t^4c_t^4+b_t^4c_t^4))}{2a_t^3b_t^3c_t^3} \ (e^{135}-e^{146}-e^{236}-e^{245}).
\end{align*}
So the system of ODEs we get if the family \eqref{eqn:phit_h7} is a solution for \eqref{eq:pgflw2} is given by
	\begin{align}\label{eq:pgflw2heigen}
	\begin{split}	
		\frac{d(f_ta_t^2)}{dt}&= -\frac{f_t^3(a_t^2(b_t^2+c_t^2)+3b_t^2c_t^2)}{2a_t^2b_t^2c_t^2} ,\\
		\frac{d(f_tb_t^2)}{dt}&=\frac{f_t^3(b_t^2(a_t^2+c_t^2)+3a_t^2c_t^2)}{2a_t^2b_t^2c_t^2},\\
		\frac{d(f_tc_t^2)}{dt}&= \frac{f_t^3(c_t^2(a_t^2+b_t^2)+3a_t^2b_t^2)}{2a_t^2b_t^2c_t^2},\\
		\frac{d(a_tb_tc_t)}{dt}&=\frac{f_t^2(2(a_tb_tc_t)^2(a_t^2+b_t^2+c_t^2)+3(a_t^4b_t^4+a_t^4c_t^4+b_t^4c_t^4))}{2a_t^3b_t^3c_t^3} .
\end{split}		
	\end{align}
	
Just as in \Cref{thm:RHFheisenberg}, we can specialize to the case when $a(t)=b(t)=c(t)=h(t)$ and $h(0)=1$, $f(0)=1$ in both \eqref{eq:pgflw1heigen} and \eqref{eq:pgflw2heigen} and get the solutions as in \Cref{thm:pgflowcCY}. These will again provide immortal solutions of the flow.

\medskip

Finally, we look at the negative gradient flow \eqref{eq:neggradflow} on the Heisenberg group. To this end, we derive the general system of ODEs for the functions $f(t),\ a(t),\ b(t)$ and $c(t)$ in \eqref{eqn:phit_h7} so that the resulting solution will give an explicit solution to the negative gradient flow in \eqref{eq:neggradflow}.

\medskip

Using \eqref{eq:Rict_h7}, \eqref{eq:TtranposeT_h7} and \eqref{eq:traceT_h7} we can compute that for $\g2_t$ in \eqref{eqn:phit_h7},

\begin{align*}
	i_{\g2_t}\left( -\Ric_t + \tr T_t(T_t)-\frac 12|T_t|^2_{g_t}\right) &=-\frac{f_t^3(11a_t^4c_t^4+2a_t^2b_t^2c_t^4-b_t^4c_t^4+2a_t^4b_t^2c_t^2+2a_t^2b_t^4c_t^2+11a_t^4b_t^4)}{8a_t^2b_t^4c_t^4}  e^{127}\\
	& \quad +\frac{f_t^3(a_t^4c_t^4-2a_t^2b_t^2c_t^4-11b_t^4c_t^4-2a_t^4b_t^2c_t^2-2a_t^2b_t^4c_t^2-11a_t^4b_t^4)}{8a_t^4b_t^2c_t^4}e^{347} \nonumber \\
	& \quad -\frac{f_t^3(11a_t^4c_t^4+2a_t^2b_t^2c_t^4+11b_t^4c_t^2+2a_t^4b_t^2c_t^2+2a_t^2b_t^4c_t^2-a_t^4b_t^4)}{8a_t^4b_t^4c_t^2}e^{567}\\
	& \quad -\frac{3f_t^2(a_t^4c_t^4+2a_t^2b_t^2c_t^4+b_t^4c_t^4+2a_t^4b_t^2c_t^2+2a_t^2b_t^4c_t^2+a_t^4b_t^4)}{aa_t^3b_t^3c_t^3} \ (e^{135}-e^{146}-e^{236}-e^{245}).
\end{align*}

\noindent	
So the system of ODEs we get for the family \eqref{eqn:phit_h7} to be a solution for \eqref{eq:neggradflow} is given by
\begin{align}\label{eq:ngflw1heigen}
	\begin{split}		
		\frac{d(f_ta_t^2)}{dt}&= -\frac{f_t^3(11a_t^4c_t^4+2a_t^2b_t^2c_t^4-b_t^4c_t^4+2a_t^4b_t^2c_t^2+2a_t^2b_t^4c_t^2+11a_t^4b_t^4)}{8a_t^2b_t^4c_t^4},\\
		\frac{d(f_tb_t^2)}{dt}&=\frac{f_t^3(a_t^4c_t^4-2a_t^2b_t^2c_t^4-11b_t^4c_t^4-2a_t^4b_t^2c_t^2-2a_t^2b_t^4c_t^2-11a_t^4b_t^4)}{8a_t^4b_t^2c_t^4},\\
		\frac{d(f_tc_t^2)}{dt}&=-\frac{f_t^3(11a_t^4c_t^4+2a_t^2b_t^2c_t^4+11b_t^4c_t^2+2a_t^4b_t^2c_t^2+2a_t^2b_t^4c_t^2-a_t^4b_t^4)}{8a_t^4b_t^4c_t^2},\\
		\frac{d(a_tb_tc_t)}{dt}&=-\frac{3f_t^2(a_t^4c_t^4+2a_t^2b_t^2c_t^4+b_t^4c_t^4+2a_t^4b_t^2c_t^2+2a_t^2b_t^4c_t^2+a_t^4b_t^4)}{8a_t^3b_t^3c_t^3}.
	\end{split}	
\end{align}

The system of ODEs in \eqref{eq:ngflw1heigen} is very general and specializing to special values of $f(t),\ a(t),\ b(t)$ and $c(t)$ will provide explicit solutions for the negative gradient flow. If $a(t)=b(t)=c(t)=h(t)$ with $h(0)=1$ and if $f(0)=1$, then we can solve the ODEs in \eqref{eq:ngflw1heigen} to get solutions as in \Cref{thm:ngflowcCY} which will provide ancient solutions of the negative gradient flow on the Heisenberg group.

\printbibliography
	
	\noindent
	(SD): Fachbereich Mathematik, Universität Hamburg, Bundesstraße 55, 20146 Hamburg, Germany.\\
	\href{mailto:shubham.dwivedi@uni-hamburg.de}{shubham.dwivedi@uni-hamburg.de}\\
	
	\noindent
	(RS): University of Münster, Einsteinstrasse 62, 48149 Münster, Germany. \\
	\href{ rsinghal@uni-muenster.de}{ rsinghal@uni-muenster.de}.

\end{document}